\documentclass[reqno]{amsart}

\pdfoutput=1

\usepackage{latexsym,amssymb,amsmath,bm,graphicx,enumitem,verbatim}
\usepackage{tikz}
\usepackage{todonotes}
\usepackage{float}

\usepackage[urlcolor=blue]{hyperref}

\usetikzlibrary{positioning,decorations.pathmorphing,fit,petri,backgrounds,decorations.pathreplacing}

\makeatletter
\@namedef{subjclassname@2020}{\textup{2020} Mathematics Subject Classification}
\makeatother

\title[Asymmetric cut and choose games]{Asymmetric cut and choose games}

\DeclareMathOperator{\restr}{\!\upharpoonright\!}

\DeclareMathOperator{\Ord}{Ord}

\newcommand{\B}{\mathbb Q}

\DeclareMathOperator{\range}{range}

\DeclareMathOperator{\dom}{dom}

\newcommand{\Pow}{P}

\DeclareMathOperator{\id}{id}
\DeclareMathOperator{\bd}{bd}

\DeclareMathOperator{\crit}{crit}
\DeclareMathOperator{\cof}{cof}

\DeclareMathOperator{\forces}{\Vdash}

\DeclareMathOperator{\GCH}{GCH}
\DeclareMathOperator{\CH}{CH}

\newcommand{\Coll}{\mathrm{Coll}}
\DeclareMathOperator{\Aut}{Aut}

\definecolor{dblue}{rgb}{0,0,0.70}
\hypersetup{
	unicode=true,
	colorlinks=true,
	citecolor=dblue,
	linkcolor=dblue,
	anchorcolor=dblue
}
\hypersetup{pdftitle = , pdfauthor = , pdfsubject= }
\hypersetup{pdfstartview=}

\author{Peter Holy}
\address{Technische Universit\"at Wien\\
Institute of discrete mathematics and geometry\\
Wiedner Hauptstraße 8–10/104, 1040 Wien\\
Austria}
\email{pholy@math.uni-bonn.de}

\author{Philipp Schlicht}
\address{School of Mathematics, University of Bristol, Fry Building, Woodland Road, Bristol, BS8 1UG, UK}
\email{philipp.schlicht@bristol.ac.uk}

\author{Christopher Turner}
\address{School of Mathematics, University of Bristol, Fry Building, Woodland Road, Bristol, BS8 1UG, UK}
\email{christopher.turner@bristol.ac.uk}

\author{Philip Welch}
\address{School of Mathematics, University of Bristol, Fry Building, Woodland Road, Bristol, BS8 1UG, UK}
\email{p.welch@bristol.ac.uk}


\thanks{We would like to thank the anonymous referee for very detailed comments.} 

\thanks{The research of the first-listed author was supported by the Italian PRIN 2017 Grant \emph{Mathematical Logic: models, sets, computability} and by the Austrian Science Fund (FWF) Elise Richter grant number V844. The first-listed author would like to thank the logic group at the University of Udine for enabling him to give a series of lectures on some of the topics of these notes early in 2022.
This project has received funding from the European Union’s Horizon 2020 research and innovation programme under the Marie Sk\l odowska-Curie grant agreement No 794020 (Project IMIC) of the second-listed author. 
The third-listed author was supported by EPSRC grant number EP/R513179/1 and also by a scholarship from the Heilbronn Institute. 
This research was funded in whole or in part by EPSRC grant number EP/V009001/1 of the second- and fourth-listed authors.  For the purpose of open access, the authors have applied a ‘Creative Commons Attribution' (CC BY) public copyright licence to any Author Accepted Manuscript (AAM) version arising from this submission. 
}

\begin{document}

\begin{abstract}
  We investigate a variety of cut and choose games, their relationship with (generic) large cardinals, and show that they can be used to characterize a number of properties of ideals and of partial orders: certain notions of distributivity, strategic closure, and precipitousness. 
\end{abstract}

\maketitle

\tableofcontents 

\newcommand{\chkfront}{\scalebox{1.6}[1.0]{$\vee$}}
\newcommand{\UU}{\mathcal U}
\newcommand{\GG}{\mathcal G}
\newcommand{\PP}{\mathcal B}

\newtheorem{fact}{Fact}[section]
\newtheorem{lemma}[fact]{Lemma}
\newtheorem{theorem}[fact]{Theorem}
\newtheorem{corollary}[fact]{Corollary}
\newtheorem{claim}[fact]{Claim}
\newtheorem{subclaim}[fact]{Subclaim}
\newtheorem{conjecture}[fact]{Conjecture}
\newtheorem{observation}[fact]{Observation}
\newtheorem{proposition}[fact]{Proposition}
\newtheorem{counterexample}[fact]{Counterexample}
\theoremstyle{definition}
\newtheorem{question}[fact]{Question}
\newtheorem{remark}[fact]{Remark}
\newtheorem{definition}[fact]{Definition}

\section{Introduction}\label{section:ulam}

Many large cardinal principles and combinatorial properties of ideals and posets have been characterised via infinite games, for instance $\omega_1$-Erd\"os cardinals by Hans-Dieter Donder and Jean-Pierre Levinski \cite{donder1989some}, completely Ramsey cardinals by Ian Sharpe and the fourth-listed author \cite{MR2817562}, remarkable cardinals by Ralf Schindler \cite{SchindlerNotiz}, $\alpha$-Ramsey cardinals by the first- and second-listed authors \cite{MR3800756}, 
completely ineffable cardinals by the first-listed author and Philipp L\"ucke \cite{pp}, 
the precipitousness of ideals \cite[Lemma 22.21]{MR1940513} and the strategic closure of posets. 
Fundamental games related to large cardinals are cut and choose games that were first introduced and studied by Fred Galvin, Jan Mycielski and Stanis\l aw Ulam in the 1960s (see \cite[Section 2]{boban}). 
In this paper, we survey and extend results on these games and some of their variants, their connections with generic large cardinals and properties of ideals, in particular with precipitousness and distributivity. 

We will restrict our attention to \emph{asymmetric} cut and choose games. 
The word asymmetric refers to the fact that in the two player games studied here, the two players, which we call \emph{Cut} and \emph{Choose}, perform different tasks: One of them presents certain partitions while the other picks elements of those partitions. Similar games in which however both players cut and choose have been studied in the literature (see \cite[Page 185]{scheepers}, \cite[Page 249]{evolution}, \cite[Page~733]{boban}).
We start by introducing what seems to be perhaps the most standard notion of asymmetric cut and choose game.

\begin{definition}\label{definition:ulam}
  Let $\kappa$ be a regular cardinal, and let $\gamma\le\kappa$ be a limit ordinal. Let $\UU(\kappa,\gamma)$ denote the following game of length $\gamma$ on $\kappa$. Initially starting with all of $\kappa$, two players, \emph{Cut} and \emph{Choose}, take turns to make moves as follows. In each move, \emph{Cut} divides a given subset of $\kappa$ into two disjoint pieces, and then \emph{Choose} answers by picking one of them. In the next round, this set is then divided by \emph{Cut} into two disjoint pieces, one of which is picked by \emph{Choose} etc. At limit stages, intersections are taken. In the end, \emph{Choose} wins in case the final intersection of their choices contains at least two distinct elements, and \emph{Cut} wins otherwise.\footnote{See the beginning of Section \ref{section:weak} for a discussion concerning the requirement that the final intersection of choices contains at least two elements.}
\end{definition}


\medskip

Let us provide the following basic observation, which shows that the consideration of winning strategies for \emph{Cut} in games of the form $\UU(\kappa,\gamma)$ is not particularly interesting. It is probably essentially due to Stephen Hechler, however has never been published, and is vaguely mentioned in a footnote to \cite[Theorem~1]{gjm}.

\begin{observation}\label{observation:cutwinsulam}
  Let $\gamma\le\kappa$ be a limit ordinal. Then, \emph{Cut} has a winning strategy in the game \,$\UU(\kappa,\gamma)$ if and only if $\kappa\le 2^{|\gamma|}$.
\end{observation}
\begin{proof}
  First, assume for a contradiction that \emph{Cut} has a winning strategy $\sigma$, however $\kappa>2^{|\gamma|}$. The strategy $\sigma$ can be identified with a full binary tree $T$ of height $\gamma$, where the root of the tree is labelled with $\kappa$, and if a node of the tree is labelled with $y$, then its immediate successor nodes (in the natural ordering of the tree, which is by end-extension) are labelled with the sets from the partition that is the response of $\sigma$ to the sequence of cuts and choices leading up to the choice of $y$, and limit nodes are labelled with the intersection of the labels of their predecessors. Since $\sigma$ is a winning strategy, the intersection of labels along any branch of $T$\footnote{For the purposes of this paper, a branch through a tree is a sequence that is increasing and downward closed in $T$ with respect to its ordering, and of length the height of $T$.} has cardinality at most one. But note that the union over all these intersections has to be $\kappa$, which clearly contradicts our assumption, for the number of branches is $2^{|\gamma|}<\kappa$.
  
  \medskip  
  
  Now assume that $\kappa\le 2^{|\gamma|}$. We may thus identify $\kappa$ with a subset $X$ of the higher Cantor space ${}^\gamma 2$.\footnote{This is the space with underlying set ${}^\gamma 2$, and with the bounded topology, that is with basic clopen sets of the form $[g]=\{f\in{}^\gamma 2\mid f\supseteq g\}$ for $g\colon\delta\to 2$ for some $\delta<\gamma$.} The winning strategy for \emph{Cut} is to increasingly partition the space ${}^\gamma 2$ (and thus also $X$) in $\gamma$-many steps using basic open sets of the form $[f]$ for functions $f\colon\xi\to 2$, with increasingly large $\xi<\gamma$. 
  In the end, the intersection of all of the sets that \emph{Choose} picked in a run of the game can clearly only contain at most one element, yielding \emph{Cut} to win, as desired.
\end{proof}

By the below remark, for a fixed $\gamma$, what is interesting is the \emph{least} $\kappa$ such that \emph{Choose} has a winning strategy in the game $\UU(\kappa,\gamma)$.

\begin{remark}\label{remark:interesting}
  If there is some cardinal $\kappa$ such that \emph{Choose} has a winning strategy in the game $\UU(\kappa,\gamma)$, then they have a winning strategy in the game $\UU(\theta,\gamma)$ whenever $\theta>\kappa$ as well: they simply pick their choices according to their intersection with $\kappa$.
\end{remark}

The games $\UU(\kappa,\gamma)$ are closely tied to large cardinals. 
If $\kappa$ is measurable, then it is easy to see that \emph{Choose} has a winning strategy for $\UU(\kappa,\omega)$, and in fact for $\UU(\kappa,\gamma)$ whenever $\gamma<\kappa$ (see Observation \ref{observation:measurable}). 
And actually, measurable cardinals are necessary in some way: If \emph{Choose} has a winning strategy in $\UU(\kappa,\omega)$, then there exists a measurable cardinal in an inner model (see Theorem \ref{theorem:ulamstrength}). 
Furthermore, variants of this game can be used to characterize weakly compact cardinals (see Observation \ref{observation:cutwinsidealulam}), various notions of distributivity (Section \ref{section:distributivity}), strategic closure of posets and precipitousness of ideals (Section \ref{section:ulamandprecipitous}). 

\medskip 

Various other interesting classes of games can be obtained from the above cut and choose games by the following adjustments, 
several of which have been studied in the set theoretic literature before. 


\begin{enumerate} 
\item 
\label{list of various games win} 
{ Winning conditions:} \vspace{2pt} 
\begin{enumerate} 
\item 
\label{list of various games win rounds} 
{\it Final requirements.} 
Instead of the requirement that the final intersection cannot have size at most $1$, this should hold in each round while the final intersection is only required to be nonempty. 
This is the weakest possible cut and choose game in the sense that it is easiest for \emph{Choose} to win. 
We study this variant in Section \ref{section:weak}. 

\item 
\label{list of various games win small} 
{\it Notions of smallness.} 
The family of subsets of $\kappa$ of size at most $1$ is replaced by an arbitrary monotone family, i.e.\ a family of subsets of $\kappa$ that is closed under subsets. 
A canonical choice is the bounded ideal $\bd_\kappa$ on $\kappa$, or other ${<}\kappa$-closed ideals on $\kappa$ that extend $\bd_\kappa$. We study such generalizations in Section \ref{section:ideal}. 
\end{enumerate} 

\item 
\label{list of various games moves} 
{ Types of moves:} \vspace{2pt} 
\begin{enumerate} 
\item 
\label{list of various games moves 1} 
{\it Partitions.} 
Each move of \emph{Cut} is a partition of $\kappa$ into a number of pieces which are disjoint only modulo a ${<}\kappa$-closed ideal $I$ in Section~\ref{section:generalizations}. 
This leads to characterizations of various notions of distributivity of $I$ in Section \ref{section:distributivity} and precipitousness of $I$ in Section \ref{section:ulamandprecipitous}. 

\item 
{\it Poset games.} 
The moves of \emph{Cut} are maximal antichains in a poset, of which \emph{Choose} picks one element. In order for \emph{Choose} to win, their choices need to have lower bounds in the poset.
This is used in Section~\ref{section:distributivity} to characterize notions of distributivity. 
\end{enumerate} 

\end{enumerate} 

\begin{remark} 
Note that poset games can have arbitrary length, and after Definition~\ref{definition:UlamonBA}, we will also briefly consider games of length ${\ge}\kappa$ as in \eqref{list of various games moves 1}. A natural extension of the games in \eqref{list of various games win rounds}  to games of length ${\geq}\kappa$, which we do not study in this paper, is obtained in \emph{filter games} by weakening the winning condition for \emph{Choose} 
to the requirement that the set of their choices generates a ${<}\kappa$-closed filter. 
Introducing \emph{delays} in this game, i.e., allowing \emph{Cut} to make $\kappa$ moves in each single round before it is \emph{Choose}'s turn to make $\kappa$-many choices, leads to characterisations of 
the $\alpha$-Ramsey cardinals defined in \cite{MR3800756} (see also \cite{nielsen-welch, foreman2020games}). 
\end{remark}

We shall provide an overview of results on the existence of winning strategies for various cut and choose games, and their connections with generic large cardinals, and combinatorial properties of ideals and posets. 
This includes a number of previously unpublished proofs, extensions of known results to more general settings and new results. 

We will show the following results concerning the above types of games. 
In Section \ref{section:weak}, 
we show that \emph{Choose} having a winning strategy in the games in \eqref{list of various games win rounds} has the consistency strength of a measurable cardinal. 
In Section \ref{section:ideal}, 
we show that certain instances of generic measurability of $\kappa$ suffice in order for \emph{Choose} to win games defined relative to ideals on $\kappa$ as in \eqref{list of various games win small}. 
In Section \ref{section:smallinaccessibles}, we show that starting from a measurable cardinal, one can force to obtain a model in which the least cardinal $\kappa$ such that \emph{Choose} wins $\UU(\kappa,\omega)$ is a non-weakly compact inaccessible cardinal. 
In Section \ref{section:distributivity}, we investigate the close connections between the existence of winning strategies for \emph{Cut} in certain cut and choose games and various notions of distributivity. In Section \ref{section:ulamandprecipitous}, we investigate connections with Banach-Mazur games on partial orders, showing in particular that these Banach-Mazur games, which will be defined in Section \ref{section:ulamandprecipitous}, are equivalent to certain cut and choose games. In Section~\ref{section:questions}, we make some final remarks and provide some open questions.

\section{The weakest cut and choose game}\label{section:weak}

Regarding Definition \ref{definition:ulam}, it may seem somewhat odd to require \emph{two} elements in the final intersection of choices in order for \emph{Choose} to win games of the form $\UU(\kappa,\gamma)$. But note that if we required only \emph{one} element in this intersection, then \emph{Choose} easily wins any of these games by fixing some ordinal $\alpha<\kappa$ in advance, and then simply picking the set that contains $\alpha$ as an element in each of their moves, for this $\alpha$ will then clearly be contained as an element in the final intersection of their choices as well. By requiring \emph{two} elements in the final intersection of their choices, this strategy is not applicable as soon as \emph{Cut} plays a partition of the form $\langle\{\alpha\},Y\rangle$.

\medskip

In this section, we will be considering canonical variants of the games $\UU(\kappa,\gamma)$. Among the cut and choose games of length $\gamma$ that we consider in this paper, these are the easiest for \emph{Choose} to win. They have (or rather, an equivalent form of them has) already been considered in unpublished work of Galvin (see \cite[Section 3, A game of Galvin]{scheepers}).

\begin{definition}
  Let $\kappa$ be a regular uncountable cardinal, and let $\gamma\le\kappa$ be a limit ordinal. Let $\UU(\kappa,{\le}\gamma)$ denote the following game of length (at most) $\gamma$ on the cardinal~$\kappa$. As in the game $\UU(\kappa,\gamma)$, starting with all of $\kappa$, players \emph{Cut} and \emph{Choose} take turns, with \emph{Cut} dividing a given subset of $\kappa$ in two, and \emph{Choose} picking one of the pieces and returning it to \emph{Cut} for their next move. \emph{Cut} wins and the game immediately ends if \emph{Choose} ever picks a singleton. At limit stages, intersections are taken. If the game lasts for $\gamma$-many stages, \emph{Choose} wins in case the final intersection of their choices is nonempty. Otherwise, \emph{Cut} wins.
\end{definition}

Note that unlike in the games $\UU(\kappa,\gamma)$, fixing one element $\alpha\in\kappa$ at the beginning of the game, and picking the set which contains $\alpha$ as an element in each of their moves is not a winning strategy in the games $\UU(\kappa,{\le}\gamma)$, since \emph{Cut} can play a partition of the form $\langle\{\alpha\},X\rangle$ at some point, so that \emph{Choose} would pick $\{\alpha\}$, and would thus immediately lose such a run.

\medskip

%

The games $\UU(\kappa,{\le}\gamma)$ behave somewhat differently with respect to the existence of winning strategies for \emph{Cut}. At least the forward direction in the following observation is attributed to unpublished work of Galvin, and independently to Hechler in~\cite{scheepers}. We do not know of any published proof of this result. For arbitrary ordinals~$\gamma$, we let $2^{<\gamma}=\sup\{2^\delta\mid\delta<\gamma$ is a cardinal$\}$.

\begin{observation}\label{cut does not win the weak Ulam game for large kappa} 
  If $\kappa$ is a regular uncountable cardinal and $\gamma<\kappa$ is a limit ordinal, then \emph{Cut} has a winning strategy in the game $\UU(\kappa,{\le}\gamma)$ if and only if $\kappa\le 2^{<\gamma}$.

\end{observation}
\begin{proof}
    Assume first that \emph{Cut} has a winning strategy $\sigma$ in the game $\UU(\kappa,{\le}\gamma)$. The strategy $\sigma$ can be identified with a binary tree $T$ of height $\gamma$, where the root of the tree is labelled with $\kappa$, and if a node of the tree is labelled with a set $X$ which is not a singleton, then its immediate successor nodes are labelled with the sets from the partition that is the response of $\sigma$ to the sequence of cuts and choices leading up to the choice of $X$, and limit nodes are labelled with the intersection of the labels of their predecessors. $\sigma$ being a winning strategy means that the intersection of labels along any branch of $T$ of length $\gamma$ is empty. Thus, for each ordinal $\alpha<\kappa$ there has to be a node labelled with $\{\alpha\}$, for this is the only reason why $\alpha$ would not appear in an intersection of choices along some branch of $T$. However there are only at most $2^{<\gamma}$-many nodes in this tree, hence $\kappa\le 2^{<\gamma}$.
    
    Now assume that $\kappa\le 2^{<\gamma}$. Let $X\subseteq{}^\gamma 2$ be such that for $y\in {}^\gamma 2$, we have $y\in X$ if and only if there is $\alpha<\gamma$ such that for all $\beta<\gamma$,
    \begin{itemize}
      \item $y(2\cdot\beta)=1\,\leftrightarrow\,\beta\ge\alpha$ and
      \item $\beta\ge\alpha\,\to\,y(2\cdot\beta+1)=1$.
    \end{itemize}
    By our assumption on $\kappa$, we may identify $\kappa$ with a subset $Y$ of the space $X$. The winning strategy for \emph{Cut} is to increasingly partition the space $Y$ in $\gamma$-many steps using sets of the form $[f]\cap Y$ for functions $f\colon\xi\to 2$ with increasingly large $\xi<\gamma$. That is, during a run of $\UU(\kappa,{\le}\gamma)$, \emph{Cut} and \emph{Choose} work towards constructing a function $F\colon\gamma\to 2$, the only possible element of the intersection of all choices of \emph{Choose}, fixing one digit in each round of the game. Now in any even round, \emph{Choose} cannot possibly pick the digit $1$, for this would correspond to picking a set $[f]\cap Y$ that is only a singleton, by the definition of the set $X$. But this means that either \emph{Cut} already wins at some stage less than $\gamma$, or that $F\colon\gamma\to 2$ is not an element of $X$. But this again means that \emph{Cut} wins, for it implies that the intersection of all choices of \emph{Choose} is in fact empty.
\end{proof}

\emph{Choose} having a winning strategy in $\UU(\kappa,{\le}\omega)$ for some cardinal $\kappa$ has the consistency strength of a measurable cardinal. A slightly weaker version of this result for the game $\UU(\kappa,\gamma)$ with essentially the same proof, that is due to Silver and Solovay, appears in \cite[Page 249]{evolution}. However, the proof that is presented there is somewhat incomplete (in particular, the argument for what would correspond to Claim \ref{wellfoundedultrapower} below is missing), and we do not know of any other published proof of this result. For this reason, even though it is just a minor adaption of a classic result, we would like to provide a complete argument for the below.

\begin{definition}
  We say that $\kappa$ is \emph{generically measurable as witnessed by the notion of forcing $P$} if in every $P$-generic extension, there is a $V$-normal $V$-ultrafilter on~$\kappa$ that induces a well-founded (generic) ultrapower of $V$. Equivalently, in every $P$-generic extension $V[G]$, there is an elementary embedding $j\colon V\to M$ with critical point $\kappa$ for some transitive $M\subseteq V[G]$.
\end{definition}

We will make use of the following standard fact. 
We include its short proof for the benefit of the reader. 

\begin{fact}\label{fact:crit}
  Assume that $U$ is a nonprincipal $V$-ultrafilter on $\kappa$ in a $P$-generic extension of the universe, that $U$ yields a wellfounded ultrapower of $V$, and that $j$ is the generic embedding induced by $U$. Let $\delta\le\kappa$ be the $V$-\emph{completeness} of $U$, that is the least $\bar\delta$ such that $U$ is not $V$-$<\bar\delta^+$-complete. Then, $\crit j=\delta$.
\end{fact}
\begin{proof}
  Pick a partition $\langle E_\xi\mid\xi<\delta\rangle\in V$ of $\kappa$ with each $E_\xi\not\in U$. Define $h\colon\kappa\to\delta$ by letting $h(\alpha)=\xi$ if $\alpha\in E_\xi$. Using that $U$ is nonprincipal, and letting $c_\alpha$ denote the function with domain $\kappa$ and constant value $\alpha$ for any ordinal $\alpha$, we have that $[c_\alpha]_U<[h]_U<[c_\delta]_U$ for every $\alpha<\kappa$. It follows that $j(\delta)>\delta$, and by the $V$-${<}\delta$-completeness of $U$, we obtain that $\crit j=\delta$, as desired.
\end{proof}

\begin{theorem}\label{theorem:ulamstrength}
  If $\gamma<\kappa$ are regular cardinals, and \emph{Choose} has a winning strategy~$\sigma$ in the game $\UU(\kappa,{\le}\gamma)$, then there exists a generically measurable cardinal below or equal to $\kappa$, as witnessed by ${<}\gamma$-closed forcing.
\end{theorem}
\begin{proof}
  Let us generically L{\'e}vy collapse $2^\kappa$ to become of size $\gamma$, by the ${<}\gamma$-closed notion of forcing $\Coll(\gamma,2^\kappa)$. In the generic extension, we perform a run of the game $\UU(\kappa,{\le}\gamma)$ with \emph{Choose} following their ground model winning strategy $\sigma$, and with the moves of \emph{Cut} following an enumeration of $\Pow(\kappa)^V$ in order-type $\gamma$. More precisely, let $\langle x_\beta\mid\beta<\gamma\rangle$ be an enumeration of $\Pow(\kappa)^V$ in our generic extension. 
  At any stage $\beta<\gamma$, assume that $\langle D_\alpha\mid\alpha<\beta\rangle$ denotes the sequence of choices of \emph{Choose} so far, and let $D_\beta'=\bigcap_{\alpha<\beta}D_\alpha$. Let \emph{Cut} play the partition $C_\beta=\langle D_\beta'\cap x_\beta,D_\beta'\setminus x_\beta\rangle$ at stage $\beta$, and let $D_\beta$ denote the response of \emph{Choose}. Note that since the L{\'e}vy collapse is ${<}\gamma$-closed, any proper initial segment of this run is in the ground model~$V$, and therefore it is possible for \emph{Choose} to apply their strategy $\sigma$ in each step. Having finished the above run of $\UU(\kappa,{\le}\gamma)$, let $U$ be the collection of all $x_\beta$'s such that $D_\beta=D_\beta'\cap x_\beta$. Equivalently, for any $x\subseteq\kappa$, $x\in U$ if and only if $D_\beta\subseteq x$ for all sufficiently large $\beta<\gamma$.
  \begin{claim}
     $U$ is a $V$-${<}\gamma$-complete, nonprincipal ultrafilter on $\Pow(\kappa)^V$.
  \end{claim}
  \begin{proof}
    It is easy to check that $U$ is an ultrafilter on $\Pow(\kappa)^V$.
    
    Let us check that $U$ is $V$-${<}\gamma$-complete. If $\delta<\gamma$, and $\langle A_i\mid i<\delta\rangle\in V$ is a sequence of elements of $U$, assume for a contradiction that $\bigcap_{i<\delta}A_i\not\in U$. Using the regularity of $\gamma$ and our above characterization of $U$, we thus find an ordinal $\eta<\gamma$ so that the intersection of choices of \emph{Choose} up to stage $\eta$ would be $\emptyset$, contradicting that \emph{Choose} follows their winning strategy $\sigma$.
    
    In order to show non-principality of $U$, note that for any $\xi<\kappa$, some $x_\beta$ is equal to $\{\xi\}$, hence $C_\beta=\langle\{\xi\},B_\beta'\setminus\{\xi\}\rangle$, and $D_\beta=B_\beta'\setminus\{\xi\}$ since $\sigma$ is a winning strategy, and therefore $\{\xi\}\not\in U$.
%
  \end{proof}
  
  \begin{claim}\label{wellfoundedultrapower}
    The generic ultrapower of $V$ by $U$ is well-founded.
  \end{claim}
  \begin{proof}
    Assume for a contradiction that this is not the case. We may thus assume that $\gamma=\omega$, for otherwise $U$ is ${<}\omega_1$-complete in a $\sigma$-closed forcing extension of the universe $V$ and therefore yields a well-founded ultrapower of $V$. Let $T$ be the tree of tuples of the form $\langle \vec f,\vec n,t\rangle$ with the following properties:
    \begin{enumerate}
      \item $t=\langle\langle A_i,B_i\rangle\mid i<k\rangle$ is a partial run of the game $\UU(\kappa,{\le}\omega)$ of length $k<\omega$ that is consistent with $\sigma$, where $A_i$ denotes the partition played by \emph{Cut}, and $B_i$ denotes the choice of \emph{Choose} at stage $i$ for every $i<k$,
      \item $\vec n=\langle n_j\mid j<l\rangle$ is a strictly increasing sequence of natural numbers for some $l\le k$, and if $l>0$ then $n_{l-1}=k-1$,
      \item $\vec f=\langle f_j\mid j<l\rangle$ is such that $f_j\colon B_{n_j}\to\Ord$ for each $j<l$, and
      \item $f_{j+1}(\alpha)<f_j(\alpha)$ for all $\alpha\in B_{n_{j+1}}$ whenever $j+1<l$.      
    \end{enumerate}
    The ordering relation on $T$ is componentwise extension of sequences, that is for $\langle\vec f,\vec n,t\rangle$ and $\langle\vec f',\vec n',t'\rangle$ both in $T$, we have $\langle\vec f,\vec n,t\rangle<_T\langle\vec f',\vec n',t'\rangle$ if $\vec f'$ extends $\vec f$, $\vec n'$ extends $\vec n$, and $t'$ extends $t$ as a sequence.
    
    \begin{subclaim}
      $T$ has a branch in $V[G]$.
    \end{subclaim}
    \begin{proof}
      Using our assumption of ill-foundedness, pick a decreasing $\omega$-sequence of ordinals in the generic ultrapower of $V$ by $U$, which are represented by functions $g_i\colon\kappa\to\Ord$ in $V$. For $i<\omega$, let $U_i=\{\alpha<\kappa\mid g_{i+1}(\alpha)<g_i(\alpha)\}$. Consider the run $\langle\langle A_i,B_i\rangle\mid i<\omega\rangle$ of the game $\UU(\kappa,{\le}\omega)$ in $V[G]$, in which \emph{Choose} plays according to $\sigma$, and in which \emph{Cut} plays based on the enumeration $\langle x_i\mid i<\omega\rangle$ of $\Pow(\kappa)^V$. We define sequences $\langle n_j\mid j<\omega\rangle$ and $\langle f_j\mid j<\omega\rangle$ inductively. Let $n_0=0$ and let $f_0=g_0\restr B_0$. Given $n_j$ and $f_j$, let $n_{j+1}$ be least above $n_j$ such that $B_{n_{j+1}}\subseteq U_j$ -- note that such $n_{j+1}$ must exist for $U_j\in U$. Let $f_{j+1}=g_{j+1}\restr B_{n_{j+1}}$. It is now straightforward to check that the sequence \[\left\langle\langle\langle f_j\mid j<l\rangle,\langle n_j\mid j<l\rangle,\langle\langle A_i,B_i\rangle\mid i\le n_{l-1}\rangle\rangle\mid l<\omega\right\rangle\] is a branch through $T$ in $V[G]$.
  \end{proof}
    By the absoluteness of well-foundedness, $T$ thus has a branch in $V$. Such a branch yields a run $\langle\langle A_i,B_i\rangle\mid i<\omega\rangle$ of the game $\UU(\kappa,{\le}\omega)$ in which \emph{Choose} follows their winning strategy, hence there is $\beta\in\bigcap_{n\in\omega}B_n\ne\emptyset$. This branch also yields a strictly increasing sequence $\langle n_i\mid i<\omega\rangle$ of natural numbers, and a sequence of functions $\langle f_i\mid i<\omega\rangle$ so that for each $i<\omega$, $f_i\colon B_{n_i}\to\Ord$, and $f_{i+1}(\alpha)<f_i(\alpha)$ whenever $\alpha\in B_{n_{i+1}}$. But then, our choice of $\beta$ yields a decreasing $\omega$-sequence $\langle f_{n_i}(\beta)\mid i<\omega\rangle$ of ordinals, which is a contradiction, as desired.
  \end{proof}
  By Fact \ref{fact:crit}, it follows that $\gamma\le\crit j\le\kappa$, and hence by the weak homogeneity of the L{\'e}vy collapse, it follows that $\crit j\le\kappa$ is generically measurable, as witnessed by ${<}\gamma$-closed forcing.\footnote{Note that since $\omega$ can never be generically measurable, this implies that in particular that if $\gamma=\omega$, $U$ will still be $V$-${<}\omega_1$-complete.}
\end{proof}

In particular thus, using standard results from inner model theory, it follows from Theorem \ref{theorem:ulamstrength} that if \emph{Choose} has a winning strategy in the game $\UU(\kappa,{\le}\omega)$, then there is an inner model with a measurable cardinal, for the existence of such an inner model follows from having a generically measurable cardinal. 
In more detail, suppose there is an elementary embedding $j\colon V\rightarrow W$ in some generic extension $V[G]$ of $V$. 
Furthermore, we may assume that there is no inner model with a measurable cardinal of order $1$.\footnote{The argument could be done similarly for other core models.} 
So the canonical least iterable structure $0^{\ddag}$ with a sharp for a measure of order $1$ does not exist (see \cite[Section 6.5]{zeman}). 
Then $0^{\ddag}$ also does not exist in $V[G]$ \cite[Lemma 6.5.6]{zeman}. 
Therefore, the core model $K$ for measures of order $0$ can be constructed in $V$ and $V[G]$ (see \cite[Section 7.3]{zeman}) and $K^V=K^{V[G]}$ by generic absoluteness of $K$ \cite[Theorem 7.4.11]{zeman}.  
In $V[G]$, the restriction $j{\upharpoonright}K$ is an elementary embedding from $K$ to a transitive class. 
Every such embedding comes from a simple iteration \cite[Theorem 7.4.8]{zeman}, i.e., there are no truncations of iterates of $K$ (see \cite[Section 4.2]{zeman}). 
Hence $K$ has a measurable cardinal. 

Note that the analogue of Remark \ref{remark:interesting} clearly applies to games of the form $\UU(\kappa,{\le}\gamma)$ as well. As another corollary of Theorem \ref{theorem:ulamstrength}, we can show that starting from a measurable cardinal, it can consistently be the case that a measurable cardinal $\kappa$ is least so that \emph{Choose} wins $\UU(\kappa,{\le}\omega)$. The same holds for $\UU(\kappa,\omega)$.

\begin{observation}\label{observation:measurableleast}
  Starting from a measurable cardinal $\kappa$, there is a model of set theory in which $\kappa$ is measurable, so that \emph{Choose} has a winning strategy in the game $\UU(\kappa,\gamma)$ whenever $\gamma<\kappa$, however \emph{Choose} doesn't have a winning strategy in the game $\UU(\lambda,\omega)$ for any $\lambda<\kappa$.
\end{observation}
\begin{proof}
  Let $U$ be a normal measurable ultrafilter on $\kappa$ and work in $L[U]$. \emph{Choose} has a winning strategy in the game $\UU(\kappa,\gamma)$ whenever $\gamma<\kappa$. Assume for a contradiction that there were some $\lambda<\kappa$ for which \emph{Choose} had a winning strategy in the game $\UU(\lambda,\omega)$. By Theorem \ref{theorem:ulamstrength}, there is a generically measurable cardinal $\nu\le\lambda$. But then, by standard inner model theory results (see the discussion before this observation), $\nu<\kappa$ would be measurable in some inner model of our universe $L[U]$. Let $u$ be a normal measurable ultrafilter on $\nu$ in that model, and consider the model $L[u]$. By classical results of Kunen (see \cite[Theorem 20.12]{MR1994835}), $L[U]$ can be obtained by iterating the measure $u$ over the model $L[u]$. But then, $u\in L[U]$ contradicts the fact that the ultrafilter $u$ could not be an element of its induced ultrapower of $L[u]$ (see \cite[Proposition 5.7(e)]{MR1994835}).
\end{proof}

%

Next, we show that we can obtain a weak version of Observation \ref{observation:cutwinsulam} for games of the form $\UU(\kappa,{\le}\gamma)$. Together with Observation \ref{cut does not win the weak Ulam game for large kappa}, this shows in particular that $\UU(\kappa,{\le}\gamma)$ is not determined when $2^{<\gamma}<\kappa\leq 2^\gamma$.

\begin{theorem}
If $\gamma\le\kappa$ is regular and $\kappa\leq 2^\gamma$, then \emph{Choose} does not have a winning strategy in the game $\UU(\kappa,{\le}\gamma)$. 
\end{theorem}
\begin{proof}
  Fix $X\subseteq{}^\gamma 2$ of size $\kappa$ that does not contain a continuous and injective image of ${}^\gamma 2$.\footnote{If $\kappa<2^\gamma$, then any $X\subseteq{}^\gamma 2$ of size $\kappa$ works. If $\kappa=2^\gamma$, then a set $X$ of size $\kappa$ which does not contain a continuous and injective image of ${}^\gamma 2$ can easily be constructed by a recursion of length $\kappa$.}    
  Let $\UU(X,{\le}\gamma)$ be the variant of $\UU(\kappa,{\le}\gamma)$ where we play on the underlying set $X$ rather than $\kappa$. Noting that these two games are equivalent, assume for a contradiction that \emph{Choose} had a winning strategy $\sigma$ for the game $\UU(X,{\le}\gamma)$. 
We consider the following quasistrategy $\tau$ for \emph{Cut}:\footnote{Unlike a strategy, which provides unique response moves, a quasistrategy provides a (nonempty) set of possible response moves (for one particular player) at each round of a game. 
It will be relevant for Claim \ref{claim choose does not win U leqgamma} below that \emph{Cut} is still free to pick the $x_i$ in odd rounds following the quasistrategy $\tau$.} 
\begin{itemize} 
\item 
In each even round $2i$, given a set $A\subseteq{}^\gamma 2$, \emph{Cut} splits it into the sets $A_0=\{x\in A\mid x(i)=0\}$ and $A_1=\{x\in A\mid x(i)=1\}$.
\item 
In each odd round $2i+1$, \emph{Cut} \emph{splits off} some singleton $\{x_i\}$, i.e., presents a partition of the form $\langle\{x_i\},X_i\rangle$. 
\end{itemize} 

Note that if \emph{Choose} wins a run of the game $\UU(X,{\le}\gamma)$ in which \emph{Cut} plays according to their quasistrategy $\tau$, then by the definition of $\tau$ at even stages, if $x\in{}^\gamma 2$ is in the intersection of choices made by \emph{Choose} in such a run, $x(i)$ has been fixed for every $i<\gamma$, that is the intersection of these choices will only have a single element. 

\begin{claim} 
\label{claim choose does not win U leqgamma} 
Suppose that $t$ is a partial play of $\UU(X,{\le}\gamma)$ of length less than $\gamma$ according to both $\sigma$ and $\tau$. Then, there are partial plays $t_0$, $t_1$ of successor length, both extending $t$ and according to both $\sigma$ and $\tau$, such that the final choices of \emph{Choose} in $t_0$ and $t_1$ are disjoint. 
\end{claim} 
\begin{proof} 
  If not, take an arbitrary run of $\UU(X,{\le}\gamma)$ extending $t$ and according to both $\sigma$ and $\tau$, such that only $x$ is in the intersection of choices along this run. 
  Now consider a different run that starts with $t$ as well, however in which \emph{Cut} splits off $\{x\}$ at the next odd stage.  If \emph{Choose} made all the same 0/1-choices at even stages in this run as before, then the intersection of their choices would now be empty, contradicting that $\sigma$ is a winning strategy for \emph{Choose}. This means that at some stage in those two runs, the respective choices of \emph{Choose} according to $\sigma$ have to be disjoint, and we may pick $t_0$ and $t_1$ to be suitable initial segments of these runs.
\end{proof} 

Using the above claim, since $\sigma$ is a winning strategy for \emph{Choose} and $\gamma$ is regular, we can construct a full binary tree $T$ of height $\gamma$ of partial plays $t$ 
such that partial plays on the same level of $T$ have the same length, and such that the final choices made by \emph{Choose} in any two such partial plays of successor length which are on the same level of $T$ will be disjoint. Let $\pi$ be an order-preserving isomorphism from ${}^{<\gamma}2\to T$, and for $a\in{}^\gamma 2$, let $\pi(a)=\bigcup\{\pi(a\restr\alpha)\mid\alpha<\gamma\}$. Since $\sigma$ is a winning strategy for \emph{Choose}, the intersection of choices from any run of $\UU(X,{\le}\gamma)$ is nonempty, and thus, using the way the quasistrategy $\tau$ was defined at even stages, this yields a continuous and injective map $f\colon {}^\gamma 2\rightarrow X$, letting $f(a)=x$ whenever $a\in{}^\gamma 2$, $b=\pi(a)$ is a branch through $T$, and $x$ is the unique element of the intersection of choices of \emph{Choose} in the run $\bigcup b$. This shows that $X$ contains a continuous and injective image of ${}^\gamma 2$, contradicting our choice of $X$. 
\end{proof}

\section{Ideal cut and choose games}\label{section:ideal}

We want to introduce a larger class of generalized cut and choose games on regular and uncountable cardinals $\kappa$, in which the winning condition is dictated by a monotone family $I$ on $\kappa$, that is a family of subsets of $\kappa$ that is closed under subsets, which in many cases will be a ${<}\kappa$-complete ideal $I\supseteq\bd_\kappa$. Before we introduce this class, let us observe that the games that we considered so far proceeded as progressions of cuts and choices, so that the chosen pieces would then be further cut into pieces. Equivalently however, we could require \emph{Cut} to repeatedly cut the starting set $\kappa$ of these games into pieces and \emph{Choose} to pick one of those pieces, in each of their moves, simply because we only evaluate intersections of choices in order to determine who wins a run of a game, so whatever happens outside of the intersection of choices made in a run of any of our games up to same stage is irrelevant for this evaluation (and every partition of $\kappa$ canonically induces a partition of any of its subsets $X$, plus every partition of some $X\subseteq\kappa$ can be \emph{extended} to a partition of $\kappa$, for example by adding all of $\kappa\setminus X$ to one of its parts). Our generalized cut and choose games will be based on the idea of \emph{Cut} repeatedly partitioning the starting set of our cut and choose games. Fix a regular uncountable cardinal $\kappa$ and a family $I$ of subsets of $\kappa$ that is monotone, i.e.\ closed under subsets, throughout this section. Let $I^+$ denote the collection of all subsets of $\kappa$ which are not elements of $I$ (\emph{$I$-positive}).

\begin{definition}\label{definition:idealulam}
  Let $X\in I^+$, and let $\gamma<\kappa$ be a limit ordinal. Let $\UU(X,I,\gamma)$ denote the following game of length $\gamma$. Starting with the set $X$, two players, \emph{Cut} and \emph{Choose}, take turns to make moves as follows. In each move, \emph{Cut} divides the set $X$ into two pieces, and then \emph{Choose} answers by picking one of them. \emph{Choose} wins in case the final intersection of their choices is $I$-positive, and \emph{Cut} wins otherwise.
  
 $\UU(X,I,{\le}\gamma)$ denotes the variant of the above game which \emph{Choose} wins just in case the intersection of their choices is $I$-positive up to all stages $\delta<\gamma$, and nonempty at the final stage $\gamma$, and \emph{Cut} wins otherwise.
  
  Let us also introduce the variant $\UU(X,I,{<}\gamma)$ for $\gamma\le\kappa$ of the above game: It proceeds in the same way for $\gamma$-many moves, however \emph{Choose} already wins in case for all $\delta<\gamma$, the intersection of their first $\delta$-many choices is $I$-positive, and \emph{Cut} wins otherwise.
\end{definition}

Note that for the games defined above, we could let them end immediately (with a win for \emph{Cut}) in case at any stage $\delta<\gamma$, the intersection of choices of \emph{Choose} up to that stage is in $I$.
Note also that if $I\subseteq J$ are monotone families on $\kappa$, $X\in J^+$, and \emph{Choose} has a winning strategy in the game $\UU(X,J,\gamma)$ for some limit ordinal $\gamma$, then they clearly also have a winning strategy in the game $\UU(X,I,\gamma)$. Moreover, if $S$ denotes the monotone family $\{\emptyset\}\cup\{\{\alpha\}\mid\alpha<\kappa\}$, then $\UU(\kappa,{\le}\gamma)$ corresponds to $\UU(\kappa,S,{\le}\gamma)$ and $\UU(\kappa,\gamma)$ corresponds to $\UU(\kappa,S,\gamma)$. We have thus in fact generalized the basic cut and choose games from our earlier sections.

\medskip

Let us start with some minor extensions of observations from Section \ref{section:ulam}. We refer to a non-principal ${<}\kappa$-complete ultrafilter on a measurable cardinal $\kappa$ as a \emph{measurable ultrafilter on $\kappa$}.

\begin{observation}\label{observation:measurable}
  If $\kappa$ is measurable, and $I$ is contained in the complement of some measurable ultrafilter $U$ on $\kappa$, then \emph{Choose} wins $\UU(X,I,{<}\kappa)$ whenever $X\in U$.
\end{observation}
\begin{proof}
  They simply win by picking their choices according to $U$.
\end{proof}

\begin{observation}\label{observation:stronglycompact}
  If $\kappa$ is $2^\kappa$-strongly compact, then \emph{Choose} wins $\UU(X,I,{<}\kappa)$ whenever $I\supseteq\bd_\kappa$ is a ${<}\kappa$-complete ideal on $\kappa$ and $X\in I^*=\{\kappa\setminus a\mid a\in I\}$.
\end{observation}
\begin{proof}
  The $2^\kappa$-strong compactness of $\kappa$ allows us to extend $I$ to a ${<}\kappa$-complete prime ideal, the complement of which thus is a measurable ultrafilter containing $X$ as an element. The result then follows from Observation \ref{observation:measurable}.
\end{proof}

We next present an observation on when \emph{Cut} wins generalized cut and choose games. This is in close correspondence to our earlier observations for games of length $\kappa$, but it also shows that \emph{Cut} not winning certain games of length $\kappa$ has large cardinal strength.\footnote{\label{footnote:WC}Item (4) below is related to the notion of a \emph{WC ideal} that was introduced by Chris Johnson~\cite{MR853844}. Johnson shows \cite[Corollary 2]{MR853844} that $\bd_\kappa$ is a WC ideal if and only if $\kappa$ is weakly compact. Using the concepts of distributivity that we will introduce in Section \ref{section:distributivity}, it is not hard to see that $I$ is a WC ideal if and only if \emph{Cut} does not win $\UU(\kappa,I,{<}\kappa)$. In fact, the following stronger statement also follows from results of Baumgartner and Johnson \cite[Paragraph after Corollary 4]{MR853844}: If $\kappa$ is weakly compact, then the weakly compact ideal $I$ on $\kappa$ is a WC ideal, and therefore \emph{Cut} does not win $\UU(\kappa,I,{<}\kappa)$.}

\begin{observation}\label{observation:cutwinsidealulam}
  Let $\gamma<\kappa$ be a limit ordinal, and let $I$ be a monotone family such that $\kappa$ cannot be written as a ${<}\kappa$-union of elements of $I$. Then, the following hold.
  \begin{enumerate}
    \item 
    \label{observation:cutwinsidealulam 1}
    \emph{Cut} wins $\UU(\kappa,I,\gamma)$ if and only if $\kappa\le 2^{|\gamma|}$.
    
    \item 
        \label{observation:cutwinsidealulam 2}
        \emph{Cut} wins $\UU(\kappa,I,{\le}\gamma)$ if and only if $\kappa\le 2^{<\gamma}$.
        
    \item 
        \label{observation:cutwinsidealulam 3}
    If $\kappa>2^{<\gamma}$, then \emph{Cut} does not win $\UU(\kappa,I,{<}\gamma)$.
    
    \item 
    \label{observation:cutwinsidealulam 4}
    If $I=\bd_\kappa$, then \emph{Cut} wins $\UU(\kappa,I,{<}\kappa)$ if and only if $\kappa$ is not weakly compact.
  \end{enumerate}
\end{observation}
\begin{proof}
  The proof of \eqref{observation:cutwinsidealulam 1} is analogous to the proof of Observation \ref{observation:cutwinsulam}, and the proof of~\eqref{observation:cutwinsidealulam 2} is analogous to the proof of Observation \ref{cut does not win the weak Ulam game for large kappa}, making use of the fact that $\kappa$ cannot be written as a ${<}\kappa$-union of elements of $I$ in the forward directions. For \eqref{observation:cutwinsidealulam 3}, assume for a contradiction that $\kappa>2^{<\gamma}$, however \emph{Cut} has a winning strategy $\sigma$ for the game $\UU(\kappa,I,{<}\gamma)$. $\sigma$ can be identified with a binary tree $T$ of height at most~$\gamma$, where the root of the tree is labelled with $\kappa$, and if a node of the tree is labelled with $y\in I^+$, then its immediate successor nodes are labelled with the sets from the partition that is the response of $\sigma$ to the sequence of cuts and choices leading up to the choice of $y$, and limit nodes are labelled with the intersection of the labels of their predecessors. If a node is labelled with a set in $I$, then it does not have any successors, and it means that \emph{Choose} has lost at such a point. $\sigma$ being a winning strategy means that $T$ has no branch of length $\gamma$. Thus, the union of all the labels of the leaves of $T$ has to be $\kappa$, which clearly contradicts our assumption on~$I$, for the number of leaves of $T$ is at most $2^{<\gamma}<\kappa$.  
  
  Regarding \eqref{observation:cutwinsidealulam 4}, assume first that $\kappa$ is weakly compact, however \emph{Cut} has a winning strategy $\sigma$ in the game $\UU(\kappa,I,{<}\kappa)$. Let $\theta$ be a sufficiently large regular cardinal, and let $M\supseteq(\kappa+1)$ be an elementary substructure of $H(\theta)$ of size $\kappa$ that is closed under ${<}\kappa$-sequences and with $\sigma\in M$. Using that $\kappa$ is weakly compact, let $U$ be a uniform ${<}\kappa$-complete $M$-ultrafilter on $\kappa$. Let us consider a run of the game $\UU(\kappa,I,{<}\kappa)$ in which \emph{Cut} follows their winning strategy $\sigma$, and \emph{Choose} responds according to $U$. This is possible for proper initial segments of such a run will be elements of $M$ by the ${<}\kappa$-closure of $M$, and hence can be used as input for $\sigma$ in $M$, yielding the individual moves of \emph{Cut} to be in $M$ as well. But since $U$ is uniform and ${<}\kappa$-complete, all choices of \emph{Choose} will be in $U$ and therefore $I$-positive. This means that \emph{Choose} wins against $\sigma$, which is our desired contradiction.
  
  In the other direction, assume that \emph{Cut} does not have a winning strategy in the game $\UU(\kappa,I,{<}\kappa)$. We verify that under the assumptions of our observation, $\kappa$ has the filter property (as in \cite[Definition 2.3]{MR3800756}) and is thus weakly compact. Let $\mathcal A=\langle A_i\mid i<\kappa\rangle$ be a collection of subsets of $\kappa$. We need to find a ${<}\kappa$-complete filter $\mathcal F$ on $\mathcal A$, that is, ${<}\kappa$-sized subsets of $\mathcal F$ need to have $\kappa$-sized intersections. At any stage $i<\kappa$, let \emph{Cut} play the partition $\langle A_i,\kappa\setminus A_i\rangle$ of $\kappa$. Since \emph{Cut} does not have a winning strategy in the game $\UU(\kappa,I,{<}\kappa)$, there is a sequence $\langle B_i\mid i<\kappa\rangle$ of choices of \emph{Choose} in such a run such that for every $\lambda<\kappa$, $\bigcap_{i<\lambda}B_i\in I^+$. But now we may clearly define our desired filter $\mathcal F$ by letting $A_i\in\mathcal F$ if $B_i=A_i$.
\end{proof}

%
%
%

Let us next observe that instead of measurability, it is sufficient for $\kappa$ to be generically measurable as witnessed by sufficiently closed forcing in order for \emph{Choose} to win cut and choose games at $\kappa$. This is a property that can be satisfied by small cardinals, and this thus shows that \emph{Choose} can win cut and choose games at small cardinals. It is well-known how to produce small cardinals that are generically measurable: For example, if $\kappa$ is measurable, as witnessed by some ultrapower embedding $j\colon V\to M$ with $\crit j=\kappa$, and for some nonzero $n<\omega$, $P$ denotes the L{\'e}vy collapse $\Coll(\aleph_{n-1},{<}\kappa)$ to make $\kappa$ become $\aleph_n$ in the generic extension, then in any $P$-generic extension $V[G]$ with $P$-generic filter $G$, $\kappa$ is generically measurable, as witnessed by the notion of forcing that is the L{\'e}vy collapse in the sense of $M[G]$ of all cardinals in the interval $[\kappa,j(\kappa))$ to become of size $\aleph_{n-1}$, which is a ${<}\aleph_{n-1}$-closed notion of forcing in $V[G]$. 
The proof for $\kappa=\aleph_1$ in \cite[Theorem 10.2]{MR2768691} works for the $\aleph_n$'s as well. 

Together with Theorem \ref{theorem:ulamstrength}, the next observation will also show that assumptions of the existence of winning strategies for \emph{Choose} in cut and choose games of increasing length form a hierarchy which is interleaved with assumptions of generic measurability, as witnessed by forcing notions with increasing closure properties. 

\begin{observation}\label{observation:genericallyulam}
  Assume that $\gamma\le\kappa$ is regular, and that $\kappa$ is generically measurable, as witnessed by some ${<}\gamma$-closed notion of forcing $P$.\footnote{The proof below can easily be adapted to the case when $P$ is only ${<}\gamma$-strategically closed.} Let $\dot U$ be a $P$-name for a uniform $V$-normal $V$-ultrafilter on $\kappa$, and let $I$ be the hopeless ideal with respect to $\dot U$, that is $I=\{Y\subseteq\kappa\mid 1\forces\check Y\not\in\dot U\}$. Then, $I\supseteq\bd_\kappa$ is a normal ideal on $\kappa$ and for any $X\in I^+$, \emph{Choose} wins $\UU(X,I,{<}\gamma)$.
\end{observation}
\begin{proof}
%
  It is straightforward to check that $I\supseteq\bd_\kappa$ is normal, using that $\dot U$ is forced to be uniform and $V$-normal.
  We will describe a winning strategy for \emph{Choose} in the game $\UU(X,I,{<}\gamma)$. At each stage $\alpha$, \emph{Choose} not only decides for a set $C_\alpha$ to actually respond with, but they also pick a condition $p_\alpha\in P$ forcing that $\check C_\alpha\in\dot U$, such that these conditions form a decreasing sequence of conditions.
  
  At stage $0$, assume that \emph{Cut} presents the partition $\langle A_0,B_0\rangle$ of $X$. Since some condition forces that $\check X\in\dot U$, \emph{Choose} may pick $C_0$ to either be $A_0$ or $B_0$, and pick a condition $p_0$ forcing that $\check C_0\in\dot U$.
  At successor stages $\alpha+1$, we proceed essentially in the same way. Assume that \emph{Cut} presents the partition $\langle A_{\alpha+1},B_{\alpha+1}\rangle$ of $C_\alpha$. Since $p_\alpha\forces\check C_\alpha\in\dot U$, \emph{Choose} may pick $C_{\alpha+1}$ to either be $A_{\alpha+1}$ or $B_{\alpha+1}$, and pick a condition $p_{\alpha+1}\le p_\alpha$ forcing that $\check C_{\alpha+1}\in\dot U$.
  
  At limit stages $\alpha<\gamma$, \emph{Cut} presents a partition $\langle A_\alpha,B_\alpha\rangle$ of $\bigcap_{\beta<\alpha}C_\beta$. Since the forcing notion $P$ is ${<}\gamma$-closed, we may let $p'_\alpha$ be a lower bound of $\langle p_\beta\mid\beta<\alpha\rangle$. Then, $p'_\alpha$ forces that $\bigcap_{\beta<\check\alpha}\check C_\beta\in\dot U$, and \emph{Choose} may pick $C_\alpha$ to either be $A_\alpha$ or $B_\alpha$, and pick a condition $p_\alpha\le p'_\alpha$ forcing that $\check C_\alpha\in\dot U$.
\end{proof}

The following result, which is also a consequence of our above results, is attributed to Richard Laver in \cite[Comment (4) after the proof of Theorem 4]{gjm}: It is consistent for \emph{Choose} to have a winning strategy in the game $\UU(\omega_2,I,\omega)$ for some uniform normal ideal $I$ on $\omega_2$, and in particular, it is consistent for \emph{Choose} to have a winning strategy in the game $\UU(\omega_2,\omega)$. By Observation \ref{observation:cutwinsulam}, $\omega_2$ will clearly be the least cardinal $\kappa$ so that \emph{Choose} has a winning strategy in the game $\UU(\kappa,\omega)$, for \emph{Cut} has a winning strategy in the game $\UU(\omega_1,\omega)$. We can now show that for either of the games $\UU(\nu,\omega)$ 
and $\UU(\nu,{\le}\omega)$, any \emph{small} successor cardinal $\nu$ of a regular and uncountable cardinal can be least so that \emph{Choose} wins. Note that the assumptions of the following observation are met in models of the form $L[U]$, when $U$ is a measurable ultrafilter on a measurable cardinal $\theta$, as we argue in the proof of Observation \ref{observation:measurableleast}.

\begin{observation}\label{observation:successorminimal}
  If $\theta$ is measurable with no generically measurable cardinals below, and given some regular and uncountable $\kappa<\theta$, then in the L{\'e}vy collapse extension by the notion of forcing $\Coll(\kappa,{<}\theta)$, making $\theta$ become $\kappa^+$, \emph{Choose} has a winning strategy in the game $\UU(\theta,\gamma)$ whenever $\gamma<\kappa$, however \emph{Choose} does not have a winning strategy in the game $\UU(\lambda,{\le}\omega)$ for any $\lambda<\theta$.
\end{observation}
\begin{proof}
  Apply the L{\'e}vy collapse forcing to make $\theta$ become $\kappa^+$, which is ${<}\kappa$-closed. Work in a generic extension for this forcing. As we argued above, $\kappa^+$ is generically measurable as witnessed by ${<}\kappa$-closed forcing. Thus by Observation \ref{observation:genericallyulam}, \emph{Choose} has a winning strategy in the game $\UU(\theta,\gamma)$ whenever $\gamma<\kappa$. Assume for a contradiction that there were some $\lambda$ with $\gamma<\lambda<\theta$ for which \emph{Choose} had a winning strategy in the game $\UU(\lambda,{\le}\gamma)$. By Theorem \ref{theorem:ulamstrength}, we obtain a generically measurable cardinal $\nu\le\lambda$. But then clearly, $\nu$ is also generically measurable in our ground model, contradicting our assumption.
\end{proof}

\section{Cut and choose games at small inaccessibles}\label{section:smallinaccessibles}

In Observation \ref{observation:measurableleast}, we observed that a measurable cardinal can be the least cardinal at which \emph{Choose} wins cut and choose games, and in Observation \ref{observation:successorminimal}, we argued that consistently, \emph{Choose} can first win cut and choose games at successors of regular cardinals. In this section, we want to show that  it is consistent for \emph{Choose} to first win cut and choose games at \emph{small} inaccessible cardinals, that is inaccessible cardinals which are not measurable, and as we will see, not even weakly compact. The key result towards this will be the following, which is an adaption of Kunen's technique \cite{kunen} of killing the weak compactness of a measurable cardinal by adding a homogeneous Suslin tree $T$, and then resurrecting measurability by forcing with $T$. Our presentation is based on the presentation of this result that is provided by Gitman and Welch in \cite[Section 6]{MR2830435}. The difference in our result below will be that we need our homogeneous Suslin tree $T$ to have additional closure properties, and that this will require us to do a little extra work at some points in the argument.

\begin{theorem}\label{theorem:winulamatsmallinaccessible}
  Given a measurable cardinal $\kappa$, and a regular cardinal $\gamma<\kappa$, one can force to obtain a model in which $\kappa$ is still inaccessible (in fact, Mahlo) but not weakly compact anymore, however generically measurable, as witnessed by ${<}\gamma^+$-closed forcing. Hence in particular, by Observation \ref{observation:genericallyulam}, \emph{Choose} wins the game $\UU(\kappa,I,\gamma)$ for some normal ideal $I\supseteq\bd_\kappa$ on $\kappa$, and thus also $\UU(\kappa,\gamma)$ and $\UU(\kappa,{\le}\gamma)$.\footnote{This also shows that under the assumption of the consistency of a measurable cardinal, an analogue of Observation \ref{observation:cutwinsidealulam} (4) does not hold for games of length less than $\kappa$.}
\end{theorem}
\begin{proof}
  We first force with a reverse Easton iteration, adding a Cohen subset to every inaccessible cardinal below $\kappa$. Let us consider the generic extension thus obtained as our ground model in the following. By well-known standard arguments (similar to those in the proof of Silver's theorem about violating the $\GCH$ at a measurable cardinal (see \cite[Theorem 21.4]{MR1940513})), adding a Cohen subset of $\kappa$ to that model will resurrect the measurability of $\kappa$ in the extension.  
  We will force to add a Suslin tree $T$ to $\kappa$ that is closed under ascending ${<}\gamma^+$-sequences, and show that $\kappa$ is generically measurable in that extension, as witnessed by forcing with that Suslin tree (with its reversed ordering), which now is a ${<}\gamma^+$-closed notion of forcing.
  
\begin{definition}
  A collection $\mathcal G$ of automorphisms of a tree $T$ \emph{acts transitively} on~$T$ if for every $a$ and $b$ on the same level of $T$, there is $\pi\in\mathcal G$ with $\pi(a)=b$.
\end{definition}

\begin{definition}
  A \emph{normal} $\alpha$-tree is a tree $T$ of height $\alpha$ with the following properties.
  \begin{itemize}
    \item Each $t\in T$ is a function $t\colon\beta\to 2$ for some $\beta<\alpha$, and $T$ is ordered by end-extension.
    \item $T$ is closed under initial segments.
    \item If $\beta+1<\alpha$ and $t\colon\beta\to 2$ is in $T$, then $t^\frown 0$ and $t^\frown 1$ are both in $T$.
    \item If $\beta<\alpha$ and $t\colon\beta\to 2$ is in $T$, then for every $\gamma$ with $\beta<\gamma<\alpha$, there is some $s\colon\gamma\to 2$ in $T$ that extends $s$ (this property is abbreviated by saying that $T$ is \emph{pruned}).
  \end{itemize}
\end{definition}

\begin{lemma}
  If $\kappa$ is inaccessible, and $\gamma<\kappa$ is a regular cardinal, then there is a ${<}\kappa$-strategically closed notion of forcing $P^\kappa_\gamma$ of size $\kappa$ that adds a $\kappa$-Suslin tree within which every increasing sequence of length at most $\gamma$ has an upper bound.
\end{lemma}
\begin{proof}
  Fix $\gamma$ and $\kappa$, and let $P^\kappa_\gamma$ be the following notion of forcing $\mathbb Q$ consisting of conditions of the form $\langle t,f\rangle$, for which the following hold:
  \begin{itemize}
    \item $t$ is a normal $(\alpha+1)$-tree that is closed under increasing unions of length at most $\gamma$, for some $\alpha<\kappa$,
    \item $\Aut(t)$ acts transitively on $t$,\footnote{The requirements on $\Aut(t)$ are needed to ensure that $\mathbb{Q}*\dot{T}$ is ${<}\kappa$-closed in Observation \ref{obs Suslin <kappa-closed}.}
     and
    \item $f\colon\nu\to\Aut(t)$ is an injective enumeration of $\Aut(t)$ for some ordinal $\nu$.
  \end{itemize}
  Conditions are ordered naturally, that is $\langle t_1,f_1\rangle\le\langle t_0,f_0\rangle$ when $t_1$ end-extends $t_0$,\footnote{That is, $t_1\supseteq t_0$ and $t_1$ restricted to the height of $t_0$ equals $t_0$.} and for all $\xi\in\dom(f_0)$, $f_1(\xi)$ extends $f_0(\xi)$.

  \begin{claim}[Strategic Closure]\label{strategicclosure}
    $\mathbb Q$ is ${<}\kappa$-strategically closed.
  \end{claim}
  \begin{proof}
    We imagine two players, Player I and Player~II taking turns for $\kappa$-many steps to play increasingly strong conditions in $\mathbb Q$. Player I has to start by playing the weakest condition of $\mathbb Q$, and is allowed to play at each limit stage of the game. The moves of Player I will be conditions denoted as $\langle t_i,f_i\rangle$, and the moves of Player II will be conditions denoted as $\langle t'_i,f'_i\rangle$. In order to show that $\mathbb Q$ is ${<}\kappa$-strategically closed, Player I has to ensure that at the end of the game, the decreasing sequence of conditions that has been produced by the above run has a lower bound in $\mathbb Q$. We will see in the argument below that it is only at limit steps when Player I has to be careful about their choice of play.
    
    Let $\langle t_0,f_0\rangle=\langle\{\emptyset\},\langle\id\rangle\rangle$ be the weakest condition of $\mathbb Q$.  Given $\langle t_i,f_i\rangle$ for some $i<\kappa$, let $\langle t'_i,f'_i\rangle\le\langle t_i,f_i\rangle$ be the response of Player~II, and let Player~I respond by any condition $\langle t_{i+1},f_{i+1}\rangle\le\langle t'_i,f'_i\rangle$ in $\mathbb Q$.
    
    At limit stages $\sigma\le\lambda$, we let $\bar t_\sigma$ be the union of the $t_\xi$'s for $\xi<\sigma$, and we let $\bar f_\sigma$ be the coordinate-wise union of the $f_\xi$'s for $\xi<\sigma$. We define the next move~$\langle t_\sigma,f_\sigma\rangle$ of Player~I as follows. In order to obtain $t_\sigma$, we add a top level to $\bar t_\sigma$ -- we do so by simply adding unions for all branches through $\bar t_\sigma$. The enumeration $f_\sigma$ is then canonically induced by $t_\sigma$ and by the $f_\xi$'s.
    
    This process can be continued for $\kappa$-many steps, showing that $\mathbb Q$ is ${<}\kappa$-stra\-te\-gi\-cally closed, as desired.
  \end{proof}
  
  Note that it is easy to extend conditions in $\mathbb Q$ to have arbitrary height below $\kappa$. A crucial property of $\mathbb Q$ is the following.
  
  \begin{claim}[Sealing]\label{sealing}
    Suppose $p\in\mathbb Q$, $\dot T$ is the canonical $\mathbb Q$-name for the generic tree added as the union of the first components of conditions in the generic filter, and $p\forces\dot A$ is a maximal antichain of $\dot T$. Then, there is $q\le p$ in $\mathbb Q$ forcing that $\dot A$ is (level-wise) bounded in $\dot T$. This means that $\dot T$ is forced to be a $\kappa$-Suslin tree.
  \end{claim}
  \begin{proof}
    Suppose $p=\langle t_0,f_0\rangle$, with $\dom(f_0)=\lambda_0$. Choose some $M\prec H(\kappa^+)$ of size less than $\kappa$ containing $\mathbb Q$, $p$, $\dot T$ and $\dot A$ as elements, such that $M$ is closed under $\gamma$-sequences, and such that $\Ord^M\cap\kappa$ is equal to some strong limit cardinal $\beta<\kappa$ of cofinality greater than $\gamma$. Let $\varphi\colon\kappa\to\kappa$ be a function in $M$ which enumerates each $\xi<\kappa$ unboundedly often. Working entirely inside of $M$, we carry out a construction in $\kappa$-many steps (so this construction only has $\beta$-many steps from the point of view of $V$). By possibly strengthening $p$, we may without loss of generality assume that there is some $a\in t_0$ such that $\langle t_0,f_0\rangle\forces\check a\in\dot A$. Let $B_0$ be any branch through $a$ in $t_0$. Let $b_0$ be the top node of $B_0$. The node $b_0$ begins the branch we will try to construct.
    
    Given $\langle t_i,f_i\rangle\in\mathbb Q$, with $\dom(f_i)=\lambda_i$, and given $b_i$, for some $i<\kappa$, let $\langle t_{i+1},f_{i+1}\rangle\in\mathbb Q$ strengthen $\langle t_i,f_i\rangle$, such that $\dom(f_{i+1})=\lambda_{i+1}$, and with the property that for every $s\in t_i$, there is $a_s\in t_{i+1}$ that is compatible with $s$ and such that $\langle t_{i+1},f_{i+1}\rangle$ forces that $\check a_s\in\dot A$. It is straightforward to obtain such a condition in $|t_0|$-many steps, making use of Claim \ref{strategicclosure}. Now, say $\varphi(i)=\rho$. If $\rho\ge\lambda_i$, let $b_{i+1}$ be a node on the top level of $t_{i+1}$ extending $b_i$. Otherwise, let $s=f_i(\rho)(b_i)$. Let $s'$ be on the top level of $t_{i+1}$ above both $s$ and $a_s$, and let $b_{i+1}=f_{i+1}(\rho)^{-1}(s')$. This will have the effect that whenever $\langle t,f\rangle\le_{\mathbb Q}\langle t_{i+1},f_{i+1}\rangle$, $\langle t,f\rangle$ will  force $\check f(\check\rho)(\check b_{i+1})$ to be above an element of $\dot A$ in this latter case.
        
    At limit stages $\sigma$, we let $\bar t_\sigma$ be the union of the $t_\xi$'s for $\xi<\sigma$, and we let $\bar f_\sigma$ be the coordinate-wise union of the $f_\xi$'s for $\xi<\sigma$. Let $b_\sigma=\bigcup_{\xi<\sigma}b_\xi$. Now, in order to obtain $t_\sigma$, we add a top level to $\bar t_\sigma$. If $\sigma$ has cofinality larger than $\gamma$, we pick this top level of $t_\sigma$ to be $\{\bigcup c[\{b_\xi\mid\xi<\sigma\}]\mid c\in\range(\bar f_\sigma)\}$. Note that since the identity map is an element of $\range(\bar f_\sigma)$, it follows in particular that $b_\sigma\in t_\sigma$. If $\sigma$ has cofinality at most $\gamma$, we pick this top level to consist of all unions of branches through $\bar t_\sigma$ (note that by the closure properties of $M$, these are the same in $M$ and in $V$, and we thus obtain an actual condition in $\mathbb Q$). Finally, $f_\sigma$ is canonically induced by $t_\sigma$ and the $f_\xi$'s in each case. It is easy to check that $\langle t_\sigma,f_\sigma\rangle$ is a condition in $\mathbb Q$ in each case, however note that having $f_\xi$ act transitively on $t_\xi$ for each $\xi<\sigma$ is needed to ensure that $t_\sigma$ is pruned.
    
    In $V$, after $\beta$-many steps, we build $q=\langle t,f\rangle=\langle t_\beta,f_\beta\rangle$ by unioning up the sequence of conditions $\langle\langle t_i,f_i\rangle\mid i<\beta\rangle$, adding a top level to $\bar t_\beta=\bigcup_{i<\beta}t_i$, and extending $\bar f_\beta$ as in the limit ordinal case above. Note that since $\beta$ has cofinality greater than $\gamma$, we will be in the case when we only include top level nodes in $t_\beta$ above \emph{certain} branches of $\bar t_\beta$.
    
    We finally need to show that $\langle t,f\rangle$ forces $\dot A$ to be bounded in $\dot T$. We will do so by showing that it forces $\dot A$ to be a maximal antichain of $\check t$ (in fact, of 
$\bar t_\beta$). Let $b$ be a branch of $\bar t_\beta$, induced by some node on the top level of $t$. This node will have to be of the form $\bigcup c[\{b_\xi\mid\xi<\beta\}]$ for $c=\bar f_\beta(\rho)=\bigcup_{\xi<\beta} f_\xi(\rho)$ for some ordinal~$\rho<\beta$. But then, using that $\varphi\in M$, and that $\beta=M\cap\kappa$, it follows that $\varphi(i)=\rho$ for unboundedly many ordinals $i<\beta$. Pick one such $i$ for which $\rho<\lambda_i$, noting that $\bigcup_{i<\beta}\lambda_i=\beta$. By our remark made at the end of the successor ordinal step of our above construction, it now follows that $\langle t,f\rangle$ forces $\check c[\{\check b_\xi\mid\xi<\check\beta\}]$ to meet $\dot A$ (within $\check t$; in fact, within $\bar t_\beta$).
  \end{proof}
  By the above claim, it is immediate that $\dot T$ is forced to be a $\kappa$-Suslin tree. By the definition of $\mathbb Q$, it is also immediate that every increasing sequence of length at most $\check\gamma$ in $\dot T$ is forced to have an upper bound in $\dot T$ (note that by its ${<}\kappa$-strategic closure, $\mathbb Q$ does not add any new ${<}\kappa$-sequences of elements of $\dot T$).
\end{proof}

\begin{observation}
\label{obs Suslin <kappa-closed} 
  If we let $\dot T$ be the canonical name for the $\kappa$-Suslin tree added by forcing with $\mathbb Q$, then $\mathbb Q*\dot T$ is equivalent to $\kappa$-Cohen forcing, where the ordering of the notion of forcing $\dot T$ is the reverse tree ordering.
\end{observation}
\begin{proof}
  It suffices to argue that $\mathbb Q*\dot T$ has a dense subset of conditions that is ${<}\kappa$-closed. Our dense set will be conditions of the form $\langle t,f,\check b\rangle$ where $b$ is a node on the top level of $t$. Given a decreasing sequence $\langle\langle t_i,f_i,\check b_i\rangle\mid i<\lambda\rangle$ of conditions in this dense set of length $\lambda<\kappa$, we may find a lower bound as in the limit stage case in the proof of Claim \ref{sealing}, with the sequence of $b_i$'s inducing a branch through the union of the $t_i$'s.
\end{proof}

  It thus follows that after forcing with $\mathbb Q*\dot T$, $\kappa$ is measurable, and thus $\mathbb Q$ forces that $\kappa$ is generically measurable as witnessed by the notion of forcing $\dot T$, which is ${<}\gamma^+$-closed, as desired. It is also clear that $\kappa$ is Mahlo after forcing with $\mathbb Q$, for otherwise it could not be measurable in the further $\dot T$-generic extension.
\end{proof}


Note that in particular, if in the starting model there are no generically measurable cardinals below $\kappa$, then in our forcing extension above, arguing as in the proof of Observation \ref{observation:measurableleast}, $\kappa$ is the least cardinal $\lambda$ such that \emph{Choose} has a winning strategy in the game $\UU(\lambda,\omega)$. The same holds for $\UU(\lambda,{\le}\omega)$.

\section{Cutting into a larger number of pieces}\label{section:generalizations}

Let us start by considering variants of cut and choose games in which we allow \emph{Cut} to cut into a larger number of pieces in each of their moves. We again fix a regular and uncountable cardinal $\kappa$ and a monotone family $I$ on $\kappa$ throughout.

\begin{definition}\label{definition:Ulamwithlargecuts}
    For any cardinal $\nu<\kappa$, and any limit ordinal $\gamma<\kappa$, we introduce the following variants $\UU_\nu(X,I,\gamma)$, $\UU_\nu(X,I,{\le}\gamma)$ and $\UU_\nu(X,I,{<}\gamma)$ of the games $\UU(X,I,\gamma)$, $\UU(X,I,{\le}\gamma)$ and $\UU(X,I,{<}\gamma)$, allowing also for $\gamma=\kappa$ in $\UU_\nu(X,I,{<}\gamma)$: In each move, \emph{Cut} is allowed to cut $X$ into up to $\nu$-many rather than just two pieces, and as before, \emph{Choose} will pick one of them. For any cardinal $\nu\le\kappa$, we also introduce variants $\UU_{{<}\nu}(X,I,\gamma)$, $\UU_{{<}\nu}(X,I,{\le}\gamma)$ and $\UU_{{<}\nu}(X,I,{<}\gamma)$: \emph{Cut} is now allowed to cut $X$ into any number of less than $\nu$-many pieces in each of their moves. The winning conditions for each of these variants are the same as for the corresponding games defined above. 
\end{definition}


\medskip

If $I$ is a ${<}\kappa$-complete ideal, then in the games $\UU_\nu(X,I,\gamma)$ and $\UU_\nu(X,I,{<}\gamma)$ above, and their variants where $\nu$ is replaced by ${<}\nu$, we could equivalently require \emph{Cut} to cut the starting set $X$ into $I$-positive sets in each of their moves: \emph{Choose} will clearly lose if they ever decide for a set in $I$, but it is also pointless for \emph{Cut} to cut off pieces in $I$, using that either our games have length less than $\kappa$, or in the case of games of length $\kappa$, the winning conditions only depend on properties of proper initial segments of its runs, and that $I$ is ${<}\kappa$-complete.

%
%

\medskip

The following generalizes Observation \ref{observation:cutwinsidealulam}, showing that it is still not very interesting to consider winning strategies for \emph{Cut} in these games.

\begin{observation}\label{observation:generalizedtrivialforcut}
  Let $\gamma<\kappa$ be a limit ordinal, let $\nu<\kappa$ be a regular cardinal, let~$I$ be a monotone family such that $\kappa$ can not be written as a ${<}\kappa$-union of elements of $I$, and let $X\in I^+$. Let $\nu^{<\gamma}=\sup\{\nu^\delta\mid\delta<\gamma$ is a cardinal$\}$. Then, the following hold.
  \begin{enumerate}
    \item \label{observation:generalizedtrivialforcut 1}
\emph{Cut} wins $\UU_\nu(\kappa,I,\gamma)$ if and only if $\kappa\le\nu^{|\gamma|}$.

    \item \label{observation:generalizedtrivialforcut 2}
\emph{Cut} wins $\UU_\nu(\kappa,I,{\le}\gamma)$ if and only if $\kappa\le\nu^{<\gamma}$.

    \item \label{observation:generalizedtrivialforcut 3}
    If $\kappa>\nu^{<\gamma}$, then \emph{Cut} does not win $\UU_\nu(\kappa,I,{<}\gamma)$.
    
    \item \label{observation:generalizedtrivialforcut 4}
    If $\kappa$ is (strongly) inaccessible, then \emph{Cut} does not win $\UU_{{<}\kappa}(\kappa,I,\gamma)$.
    
    \item \label{observation:generalizedtrivialforcut 5}
    If $\kappa$ is weakly compact, then \emph{Cut} does not win $\UU_{{<}\kappa}(\kappa,\bd_\kappa,{<}\kappa)$.
  \end{enumerate}
\end{observation}
\begin{proof}
  The proofs of \eqref{observation:generalizedtrivialforcut 1}, \eqref{observation:generalizedtrivialforcut 2}, \eqref{observation:generalizedtrivialforcut 3} and \eqref{observation:generalizedtrivialforcut 5} are analogous to those in Observation \ref{observation:cutwinsidealulam}. The argument for \eqref{observation:generalizedtrivialforcut 4} is a minor adaption of that for \eqref{observation:generalizedtrivialforcut 3}.
  \end{proof}

Observation \ref{observation:genericallyulam} easily generalizes to the following, using that in the notation of that observation, $\dot U$ is forced to be $V$-${<}\kappa$-complete.

\begin{observation}\label{observation:genericallyulam2}
  Assume that $\gamma\le\kappa$ is regular, and that $\kappa$ is generically measurable, as witnessed by some ${<}\gamma$-closed notion of forcing $P$. Let $\dot U$ be a $P$-name for a $V$-normal $V$-ultrafilter on $\kappa$, and let $I$ be the hopeless ideal with respect to $\dot U$. Then, for any $X\in I^+$, \emph{Choose} has a winning strategy in the game $\UU_{{<}\kappa}(X,I,{<}\gamma)$.
\end{observation}

We also want to define cut and choose games on a cardinal $\kappa$ where \emph{Cut} can cut into $\kappa$-many pieces. A little bit of care has to be taken in doing so however. One thing to note is that we do have to require \emph{Cut} to actually cut into $I$-positive pieces, for otherwise, given that $I$ contains all singletons, they could cut any set $X$ into singletons in any of their moves, making it impossible for \emph{Choose} to win. Another observation is that if $I\supseteq\bd_\kappa$ is ${<}\kappa$-complete, then any disjoint partition $W$ of an $I$-positive set $X$ into less than $\kappa$-many $I$-positive sets is maximal: there cannot be an $I$-positive $A\subseteq X$ such that for any $B\in W$, $A\cap B\in I$. This is clearly not true anymore for partitions of size $\kappa$. However, as the following observation shows, in many cases, maximality is needed in order for such cut and choose games to be of any interest.

\begin{observation}
  If $I\supseteq\bd_\kappa$ is ${<}\kappa$-complete and has the property that any $I$-positive set can be partitioned into $\kappa$-many disjoint $I$-positive sets,\footnote{Note that this is the case for example if $I$ is the bounded or the nonstationary ideal.} and the game $\UU_\kappa(X,I,\gamma)$ were defined as the games $\UU_\nu(X,I,\gamma)$ in Definition \ref{definition:Ulamwithlargecuts}, however letting $\nu=\kappa$ while additionally requiring \emph{Cut} to always provide partitions into $I$-positive pieces, and $X\in I^+$, then \emph{Cut} has a winning strategy in the game $\UU_\kappa(X,I,\omega)$.
\end{observation}
\begin{proof}
  Write $X$ as a disjoint union of $I$-positive sets $X_i$ for $i<\omega$. At any stage $n<\omega$, let \emph{Cut} play a disjoint partition $\langle A^n_\alpha\mid\alpha<\kappa\rangle$ of $X$ into $I$-positive sets such that each $A^n_\alpha$ contains exactly one element of $X_n$, and such that $A^n_\alpha\cap X_m\in I^+$ whenever $m>n$.\footnote{Since $X_n\cap A^n_\alpha\in I$ for every $\alpha<\kappa$, the partition $\langle A^n_\alpha\mid\alpha<\kappa\rangle$ is not maximal.} \emph{Choose} has to pick some $B^n=A^n_\alpha$. Let $A^n_\alpha\cap X_n=\{\alpha_n\}$. 
  Note that the above defines a strategy for \emph{Cut} which ensures that for any $i<\omega$, $X_i\cap\bigcap_{n<\omega}B^n$ contains at most one element, and hence the intersection $\bigcap_{n<\omega}B^n$ of choices of \emph{Choose} is countable, showing this strategy to be a winning strategy for \emph{Cut}, as desired.\footnote{With a little more effort, it is in fact possible to provide a strategy for \emph{Cut} which ensures that the intersection of choices of \emph{Choose} is empty. This makes use of the fact that for any ordinal~$\alpha$, \emph{Cut} can play to ensure that in order to have a chance of winning, \emph{Choose} has to decide for a set that does not contain $\alpha$ as an element within a finite number of moves.}
\end{proof}

We will need the following.

\begin{definition}
  Let $I$ be a monotone family on a regular and uncountable cardinal~$\kappa$.
  \begin{itemize}
    \item If $X\in I^+$, then an \emph{$I$-partition} of $X$ is a maximal collection $W\subseteq\Pow(X)\cap I^+$ so that $A\cap B\in I$ whenever $A,B\in W$ are distinct.
    \item An $I$-partition $W$ is \emph{disjoint} if any two of its distinct elements are.
  \end{itemize}
\end{definition}

In the light of the above, we now define cut and choose games in which \emph{Cut} can cut into $\kappa$-many, or even more pieces in each of their moves as follows.

%

\begin{definition}\label{definition:Ulamwithlargestcuts}
    Let $\kappa$ be a regular uncountable cardinal, let $I$ be a monotone family on $\kappa$, let $\gamma<\kappa$ be a limit ordinal, let $X\in I^+$, and let $\nu$ be a cardinal, or $\nu=\infty$.
    \begin{itemize}
      \item $\GG_\nu(X,I,\gamma)$ denotes the variant of the game $\UU_\nu(X,I,\gamma)$ where in each move, \emph{Cut} may play an $I$-partition of size at most $\nu$ of $X$, or of arbitrary size if $\nu=\infty$, and \emph{Choose} has to pick one of its elements. \emph{Choose} wins in case the intersection of all of their choices is $I$-positive, and \emph{Cut} wins otherwise.
      \item In a similar fashion (using $I$-partitions rather than disjoint partitions), we also define games $\GG_\nu(X,I,{\le}\gamma)$ and $\GG_\nu(X,I,{<}\gamma)$ as variants of $\UU_\nu(X,I,{\le}\gamma)$ and $\UU_\nu(X,I,{<}\gamma)$, allowing also for $\gamma=\kappa$ for the latter.
      \item If $\nu$ is a cardinal, we also define games $\GG_{{<}\nu}(X,I,\gamma)$, $\GG_{{<}\nu}(X,I,{\le}\gamma)$ and $\GG_{{<}\nu}(X,I,{<}\gamma)$ in the obvious way.
    \end{itemize}
\end{definition}

By the below observation, these games actually generalize the games that we introduced in Definition \ref{definition:Ulamwithlargecuts} above.

\begin{observation}\label{observation:stargeneralizes}
  If $I$ is a monotone family, then the $\UU$-games introduced in Definition \ref{definition:Ulamwithlargecuts} are equivalent to their corresponding $\GG$-games introduced in Definition~\ref{definition:Ulamwithlargestcuts}, that is, for any choice of parameters $X$, $I$, $\nu$ and $\gamma$ that are suitable for Definition \ref{definition:Ulamwithlargecuts}, the games $\GG_\nu(X,I,\gamma)$ and $\UU_\nu(X,I,\gamma)$ are equivalent, the games $\GG_\nu(X,I,{\le}\gamma)$ and $\UU_\nu(X,I,{\le}\gamma)$ are equivalent etc. 
\end{observation}
\begin{proof}
  We only treat the equivalence between games of the form $\GG_\nu(X,I,\gamma)$ and $\UU_\nu(X,I,\gamma)$ when $\nu<\kappa$ is a cardinal and $\gamma<\kappa$ is a limit ordinal, for the other equivalences are analogous. Making use of the comments after Definition \ref{definition:Ulamwithlargecuts}, if \emph{Cut} wins $\UU_\nu(X,I,\gamma)$, then \emph{Cut} wins $\GG_\nu(X,I,\gamma)$, because every disjoint partition of an $I$-positive set $X$ into less than $\kappa$-many $I$-positive sets is an $I$-partition of $X$. Analogously, if \emph{Choose} wins $\GG_\nu(X,I,\gamma)$, then \emph{Choose} wins $\UU_\nu(X,I,\gamma)$.
  
\medskip  

  Given an $I$-partition $W=\{w_\alpha\mid\alpha<\theta\}$ of some set $X\in I^+$, with $\theta\le\kappa$, we call $W'=\{w'_\alpha\mid\alpha<\theta\}$ a \emph{full disjointification} of $W$ in case $w'_0=w_0\cup (X\setminus \bigcup W)$ and $w'_\alpha=w_\alpha\setminus\bigcup_{\bar\alpha<\alpha}w_{\bar\alpha}$ for every nonzero $\alpha<\theta$.\footnote{Note that full disjointifications of an $I$-partition $W$ are not unique -- they are only determined modulo an enumeration of $W$.
In the following, let us fix some canonical choice of full disjointification for all $I$-partitions $W$, and let us refer to those as \emph{the} full disjointification of $W$.} 
Then $W'$ is a partition of $X$ 
and moreover, $w'_\alpha\subseteq w_\alpha$ for every nonzero $\alpha<\theta$. 

  \medskip
  
Suppose that $\sigma$ is a winning strategy for \emph{Cut} in $\GG_\nu(X,I,\gamma)$. 
To define a winning strategy $\tau$ for \emph{Cut} in $\UU_\nu(X,I,\gamma)$, we use an auxiliary run of $\GG_\nu(X,I,\gamma)$ in which \emph{Cut} plays according to $\sigma$. 
Given a move $W_\alpha$ of \emph{Cut} in round $\alpha<\gamma$ in $\GG_\nu(X,I,\gamma)$, we let $\tau$ perform two consecutive moves in $\UU_\nu(X,I,\gamma)$. 
The first one is the full disjointification $W_\alpha'$ of $W_\alpha$. 
The second one splits $X$ into $\bigcup W_\alpha$ and $X\setminus \bigcup W_\alpha$. 
\emph{Choose} first picks an element $Y_\alpha$ of $W_\alpha'$ and then they have to pick $\bigcup W_\alpha$. 
By the definition of full disjointifications, there is some $X_\alpha \in W_\alpha$ with $Y_\alpha\cap\bigcup W_\alpha\subseteq X_\alpha$. 
We let \emph{Choose} play such an $X_\alpha$ in $\GG_\nu(X,I,\gamma)$, and \emph{Cut} again responds in the next round by using $\sigma$. 
Since $\sigma$ is a winning strategy for \emph{Cut}, it follows that $\bigcap_{\alpha<\gamma}(Y_\alpha\cap \bigcup W_\alpha) \subseteq \bigcap_{\alpha<\gamma}X_\alpha\in I$, and therefore that $\tau$ is a winning strategy for \emph{Cut}, as desired. 

\medskip

We now argue that a winning strategy $\sigma$ for \emph{Choose} in $\UU_\nu(X,I,\gamma)$ yields a winning strategy $\tau$ for \emph{Choose} in the game $\GG_\nu(X,I,\gamma)$, making use of an auxiliary run of $\UU_\nu(X,I,\gamma)$ in which \emph{Choose} plays according to $\sigma$. Given a move $W_\alpha$ of \emph{Cut} in the game $\GG_\nu(X,I,\gamma)$, we let \emph{Cut} perform two consecutive moves in the game $\UU_\nu(X,I,\gamma)$: The first one is the full disjointification $W_\alpha'$ of $W_\alpha$, and the second one splits $X$ into $\bigcup W_\alpha$ and $X\setminus\bigcup W_\alpha$. The strategy $\sigma$ will pick some element $Y_\alpha\in W_\alpha'$, and then decides for $\bigcup W_\alpha$. We let the next move of \emph{Choose} according to $\tau$ be some $X_\alpha\in W_\alpha$ for which $Y_\alpha\cap\bigcup W_\alpha\subseteq X_\alpha$.
Since $\sigma$ is a winning strategy for \emph{Choose}, it follows that $\bigcap_{\alpha<\gamma}(Y_\alpha\cap \bigcup W_\alpha) \subseteq \bigcap_{\alpha<\gamma}X_\alpha\in I^+$, and therefore that $\tau$ is a winning strategy for \emph{Choose}, as desired. 
\end{proof}

Up to some point, increasing the possible size $\nu$ of $I$-partitions that \emph{Cut} may play actually does not make a difference (in terms of the existence of winning strategies for either player) for our generalized cut and choose games of the form $\GG_\nu(X,I,\gamma)$ or $\GG_\nu(X,I,{<}\gamma)$. This will folllow as a special case of Theorem \ref{theorem:increasedsplitting} below, noting that if $\theta$ is a cardinal and $I$ is a ${<}\theta^+$-complete ideal on a cardinal $\kappa$, then the partial order $\mathcal P(\kappa)/I$ is a ${<}\theta^+$-complete Boolean algebra.



\section{Poset games and distributivity}\label{section:distributivity}

A very natural further generalization is to consider analogues of the above games played on posets. On Boolean algebras, such games of length $\omega$ were considered by Boban Veličković \cite{boban}, and such games of arbitrary length were considered by Natasha Dobrinen \cite{dobrinen}, who also mentions a generalization to partial orders.
We assume that each poset $\B$ has domain $Q$ and a maximal element $1_\B$.

\begin{definition}\label{definition:UlamonBA}
    If $\B$ is a poset with $X\in Q$, $\gamma$ is a limit ordinal, and $\nu$ is a cardinal, or $\nu=\infty$, $\GG_\nu(X,\B,\gamma)$ denotes the game of length $\gamma$ in which players \emph{Cut} and \emph{Choose} take turns, where in each move, \emph{Cut} plays a maximal antichain of $\B$ below $X$ of size at most $\nu$, or of arbitrary size if $\nu=\infty$, and \emph{Choose} responds by picking one of its elements. \emph{Choose} wins in case the sequence of all of their choices has a lower bound in $\B$, and \emph{Cut} wins otherwise. We also introduce obvious variants with ${<}\nu$ and/or ${<}\gamma$ in place of $\nu$ and~$\gamma$ respectively -- if the final parameter is of the form ${<}\gamma$, we only ask for lower bounds in~$\B$ for all proper inital segments of the sequence of their choices in order for \emph{Choose} to win.
\end{definition}

Let $\kappa$ be a regular uncountable cardinal, and let $I$ be a ${<}\kappa$-complete ideal on $\kappa$. It is easily observed that for $X\in I^+$, any limit ordinal $\gamma<\kappa$, and any cardinal $\nu$, or $\nu=\infty$, the games $\GG_\nu(X,I,\gamma)$ and $\GG_\nu([X]_I,\Pow(\kappa)/I,\gamma)$ are essentially the same game (and are in particular equivalent), as are $\GG_\nu(X,I,{<}\gamma)$ and $\GG_\nu([X]_I,\Pow(\kappa)/I,{<}\gamma)$. But note that Definition \ref{definition:UlamonBA} can also be taken to provide a natural definition of $\GG_\nu(X,I,\gamma)$, and its variants with ${<}\nu$ and/or ${<}\gamma$, which also works for $\gamma\ge\kappa$: We could take them to be $\GG_\nu([X]_I,\Pow(\kappa)/I,\gamma)$ and its variants, and we observe that this corresponds to requiring the existence of an $I$-positive set that is $I$-almost contained in every choice of \emph{Choose} in order for \emph{Choose} to win, rather than an $I$-positive intersection of those choices, in Definition \ref{definition:Ulamwithlargestcuts}.

\medskip

We first want to show a result that we already promised (for games with respect to ideals) in Section \ref{section:generalizations}, namely that up to some point, increasing the possible size of partitions provided by \emph{Cut} still yields equivalent games. Given a cardinal $\theta$, we say that a partial order $\B$ is \emph{${<}\theta$-complete} in case it has suprema and infima for all of its subsets of size less than $\theta$, under the assumption that those subsets have a lower bound for the latter.

\begin{theorem}\label{theorem:increasedsplitting}
  Let $\gamma$ and $\nu$ be cardinals, let $\beta<\gamma$ be a cardinal, let $\B$ be a separative partial order with domain $Q$, and let $q\in Q$.
  \begin{enumerate}
    \item\label{theorem:increasedsplitting 1} If $\B$ is ${<}(\nu^\beta)^+$-complete, then $\GG_\nu(X,\B,\gamma)$ and $\GG_{\nu^{\beta}}(X,\B,\gamma)$ are equivalent, as are $\GG_\nu(X,\B,{<}\gamma)$ and $\GG_{\nu^\beta}(X,\B,{<}\gamma)$.\footnote{Note that if for some $\beta<\gamma$, we have $\nu^\beta=\nu^{<\gamma}$, then this means that the games $\GG_\nu(X,\B,\gamma)$ and $\GG_{\nu^{<\gamma}}(X,\B,\gamma)$ are equivalent. This is the case in particular if $\gamma=\beta^+$ is a successor cardinal.}
    \item 
    \label{theorem:increasedsplitting 2}
    If $\gamma$ is a limit cardinal, $\nu^\delta<\nu^{<\gamma}$ whenever $\delta<\gamma$,\footnote{By \eqref{theorem:increasedsplitting 1}, we do not need this assumption in case $\B$ is ${<}(\nu^{<\gamma})^+$-complete.} and $\B$ is ${<}(\nu^{<\gamma})$-complete, then $\GG_\nu(X,\B,\gamma)$ and $\GG_{<(\nu^{<\gamma})}(X,\B,\gamma)$ are equivalent, as are the games $\GG_\nu(X,\B,{<}\gamma)$ and $\GG_{<(\nu^{<\gamma})}(X,\B,{<}\gamma)$.
    \item 
        \label{theorem:increasedsplitting 3}
        In particular, if $\gamma$ is a strong limit cardinal, and $\B$ is ${<}\gamma$-complete, then the games $\GG_2(X,\B,{<}\gamma)$ and $\GG_{{<}\gamma}(X,\B,{<}\gamma)$ are equivalent.
\end{enumerate}
\end{theorem}
\begin{proof}
  The idea of the arguments for the above is that we may simulate a single move of \emph{Cut}, in the games where they are allowed to play larger antichains, by less than $\gamma$-many moves in the corresponding games where they are only allowed to play antichains of size at most $\nu$ in each of their moves. 
  Let us go through some of the details of one of those equivalences in somewhat more detail. For example, let us assume that in \eqref{theorem:increasedsplitting 1}, \emph{Cut} has a winning strategy in the game $\GG_{\nu^\beta}(X,\B,\gamma)$. Assume that in one of their moves, they play a maximal antichain of $\B$ below $X$ of the form $W=\{x_r\mid r\in {}^\beta\nu\}$. Let \emph{Cut} make $\beta$-many moves in the game $\GG_\nu(X,\B,\gamma)$, playing maximal antichains $W_i$ below $X$ for $i<\beta$, with $W_i=\{w_i^j\mid j<\nu\}$ such that $w_i^j=\sup\{x_r\mid r(i)=j\}$ for all $j<\nu$. Let $r\colon\beta\to\nu$ be such that \emph{Choose} picks $w_i^{r(i)}$ in their $i^\textrm{th}$ move, for each $i<\beta$. Now clearly $x_r\le w_i^{r(i)}$ for each $i<\beta$, hence $x_r\le\inf\{w_i^{r(i)}\mid i<\beta\}$. Note that for $r\ne r'\in{}^\beta\nu$, $\inf\{w_i^{r(i)}\mid i<\beta\}$ and $\inf\{w_i^{r'(i)}\mid i<\beta\}$ are incompatible, and hence the collection of these infima for different $r'\in{}^\beta\nu$ forms a maximal antichain of $\B$ below $X$. Thus, by the separativity of $\B$, it follows that in fact $x_r=\inf\{w_i^{r(i)}\mid i<\beta\}$. Let \emph{Choose} respond to $W$ by picking $x_r\in W$. \emph{Cut} will win this run of the game $\GG_{\nu^\beta}(X,\B,\gamma)$ for they are using their winning strategy, but then they will also win the above run of $\GG_{\nu}(X,\B,\gamma)$, for the responses of \emph{Choose} in this run will be cofinal in the sequence of their corresponding responses in the run of $\GG_{\nu^\beta}(X,\B,\gamma)$, and thus the set of responses of \emph{Choose} in either game will not have a lower bound. We have thus produced a winning strategy for \emph{Cut} in the game $\GG_\nu(X,I,\gamma)$ in this way, as desired.
  
  The remaining arguments for \eqref{theorem:increasedsplitting 1} are very similar to the above. Item \eqref{theorem:increasedsplitting 2} follows directly from the argument for \eqref{theorem:increasedsplitting 1}, and \eqref{theorem:increasedsplitting 3} is an immediate consequence of \eqref{theorem:increasedsplitting 2}.
\end{proof}

\medskip

Let us recall and introduce two notions of distributivity for posets.

\begin{definition}[Distributivity]\label{definition:distributivity}
  Let $\B$ be a poset with underlying set $Q$, let $\gamma$ be a limit ordinal and let~$\nu$ be a regular cardinal, or $\nu=\infty$.
  \begin{itemize}
    \item For $X\in Q$, $\B$ is \emph{uniformly $({<}\gamma,\nu)$-distributive with respect to $X$} if whenever $\langle W_\alpha\mid\alpha<\gamma\rangle$ is a sequence of maximal antichains of $\B$ below $X$, each of size at most $\nu$, or of arbitrary size in case $\nu=\infty$, then there is a sequence $\langle X_\alpha\mid\alpha<\gamma\rangle$ of conditions so that for each $\alpha<\gamma$, $X_\alpha\in W_\alpha$ and the sequence  $\langle X_\beta\mid\beta<\alpha\rangle$ has a lower bound in $\B$.     
    We call such a sequence $\langle X_\alpha\mid\alpha<\gamma\rangle$ a \emph{branch through $\langle W_\alpha\mid\alpha<\gamma\rangle$}.
    \item The poset $\B$ is \emph{uniformly $({<}\gamma,\nu)$-distributive} if it is uniformly $({<}\gamma,\nu)$-distributive with respect to $X$ for every $X\in Q$.
    \item For $X\in Q$, $\B$ is \emph{$(\gamma,\nu)$-distributive with respect to $X$} if whenever $\langle W_\alpha\mid\alpha<\gamma\rangle$ is a sequence of maximal antichains of $\B$ below $X$, each of size at most~$\nu$, or of arbitrary size in case $\nu=\infty$, then there is a sequence $\langle X_\alpha\mid\alpha<\gamma\rangle$ so that for each $\alpha<\gamma$, $X_\alpha\in W_\alpha$ and $\{X_\alpha\mid\alpha<\gamma\}$ has a lower bound in $\B$. We call such a sequence $\langle X_\alpha\mid\alpha<\gamma\rangle$ a \emph{positive branch through $\langle W_\alpha\mid\alpha<\gamma\rangle$}.
    \item The poset $\B$ is \emph{$(\gamma,\nu)$-distributive} if it is $(\gamma,\nu)$-distributive with respect to~$X$ for every $X\in Q$.\footnote{Note that there are two further common versions of $(\gamma,\nu)$-distributivity, either in terms of being able to swap the order of certain infinitary conjunctions of disjunctions, or in terms of not adding new functions from $\gamma$ to $\nu$ when forcing, see for example \cite[Page 12]{jech}. The former is easily seen to lead to an equivalent notion in the case of Boolean algebras, and the same is true for the latter in the case of complete Boolean algebras.}
    \item Let $I$ be an ideal on a regular and uncountable cardinal $\kappa$. We say that $I$ is \emph{$(\gamma,\nu)$-distributive} or \emph{uniformly $({<}\gamma,\nu)$-distributive} if the poset $\Pow(\kappa)/I$ is.
  \end{itemize}
\end{definition}

For complete Boolean algebras $\B$, it is easy to see that $(\gamma,\nu)$-distributivity implies $(\gamma,\nu^\gamma)$-distributivity, 
since adding no new functions from $\gamma$ to $\nu$ by forcing with $\B$ is clearly equivalent to adding no new functions from $\gamma$ to $\nu^\gamma$.\footnote{This was observed for $\nu=2$ in \cite[Lemma 1.60]{Sobot}.} 
The following is a version of this observation with weaker completeness assumptions that seems to require a different kind of argument. This lemma and its proof are closely related to Theorem~\ref{theorem:increasedsplitting}.

\begin{lemma}
\label{lemma: eq distributive} 
If $\B$ is a $(\delta,\nu)$-distributive poset, where $\delta$, $\nu$ and $\gamma\leq\delta$ are cardinals, then the following statements hold: 
    \begin{enumerate} 
    \setcounter{enumii}{1} 
      \item 
      \label{lemma: eq distributive 1} 
      If $\B$ is ${<}(\nu^\gamma)^+$-complete, then $\B$ is $(\delta,\nu^\gamma)$-distributive.   
      
	\item 
	\label{lemma: eq distributive 2} 
	If $\B$ is ${<}(\nu^{<\gamma})^+$-complete, then $\B$ is $(\delta,\nu^{<\gamma})$-distributive. 
	
	\item 
	\label{lemma: eq distributive 3} 
	If $\B$ is a ${<}(\nu^{<\gamma})$-complete Boolean algebra and $\nu^\beta<\nu^\gamma$ for all $\beta<\gamma$, then $\B$ is $(\delta,\nu^{<\gamma})$-distributive. 
	
      \end{enumerate} 
      In analogy to the above, uniform $({<}\delta,\nu)$-distributivity implies higher levels of uniform distributivity as well. 
      \end{lemma} 
\begin{proof} 
\eqref{lemma: eq distributive 1}: 
Suppose that $\langle W^j \mid j<\delta\rangle$ is a sequence of maximal antichains in~$\B$, each of size ${\leq}\nu^\gamma$. 
For each $j<\delta$, we define $\langle W^j_i \mid i<\gamma\rangle$ as in the proof of Theorem \ref{theorem:increasedsplitting}. 
Since $\B$ is $(\delta,\nu)$-distributive, there exists a positive branch through $\langle W^j_i \mid {\prec} i,j {\succ} \in\gamma\rangle$ with a lower bound $p$, where ${\prec} i,j {\succ}$ denotes the standard pairing function applied to $i$ and $j$. 
As in the proof of Theorem \ref{theorem:increasedsplitting}, $p$ induces a positive branch through $\langle W^j \mid j<\delta\rangle$.

\eqref{lemma: eq distributive 2}: 
Suppose that $\langle W^j \mid j<\delta\rangle$ is a sequence of maximal antichains in $\B$, each of size ${\leq}\nu^{<\gamma}$. 
Fix a cofinal sequence $\langle \gamma_i \mid i<\cof(\gamma) \rangle$ in $\gamma$. 
For each $j<\delta$, we partition $W^j$ into subsets $\langle W^{j,i}\mid i<\cof(\gamma)\rangle$ such that $W^{j,i}$ has size ${\leq}\nu^{\gamma_i}$ for each $i<\cof(\gamma)$. 
We can extend each $W^{j,i}$ to a maximal antichain $\bar{W}^{j,i}$ by adding a single condition, namely $\sup(W^j\setminus W^{j,i})$, since $\B$ is ${<}(\nu^{<\gamma})^+$-complete.  
As in the proof of \eqref{lemma: eq distributive 1}, we then replace each $\bar{W}^{j,i}$ by a sequence $\langle \bar{W}^{j,i}_k \mid k<\gamma_i\rangle$ such that $\bar{W}^{j,i}_k$ has size ${\leq}\nu$. 
Let $\langle\tilde W_l\mid l<\delta\rangle$ enumerate all the $\bar W^{j,i}_k$ in order-type $\delta$.
Since $\B$ is $(\delta,\nu)$-distributive, 
there exists a positive branch through $\langle \tilde W_l \mid l<\delta\rangle$. This is easily seen to induce a positive branch through $\langle W^j \mid j<\delta \rangle$, as required. 

\eqref{lemma: eq distributive 3}: 
We proceed as in the proof of \eqref{lemma: eq distributive 2}, except when $W^{j,i}$ is extended to a maximal antichain $\bar{W}^{j,i}$: 
Since $|W^{j,i}| \leq \nu^{\gamma_i}<\nu^{<\gamma}$ and $\B$ is ${<}(\nu^{<\gamma})$-complete, $\sup(W^{j,i})$ exists, and thus, using that $\B$ is a Boolean algebra, also $\sup(W^j\setminus W^{j,i})=\neg\sup(W^{j,i})$ exists. 
\end{proof}

For Boolean algebras, the case $\gamma=\omega$ in \eqref{theorem:ulamanddistributivity 1} and \eqref{theorem:ulamanddistributivity 2} below was proved by 
Thomas Jech in \cite[Theorem~2]{jech} (and \eqref{theorem:ulamanddistributivity 3} is nontrivial only for uncountable $\gamma$). 
A more general result for arbitrary cardinals $\gamma$ was then 
shown by 
Dobrinen in \cite[Theorem~1.4]{dobrinen}. In the theorem below, \eqref{theorem:ulamanddistributivity 1} and \eqref{theorem:ulamanddistributivity 2}--\eqref{theorem:ulamanddistributivity 2b} are essentially due to Dobrinen. We will present a somewhat different and simpler argument for these, and furthermore present additional results which partially answer a question of Dobrinen \cite[paragraph after Theorem 1.4]{dobrinen} by showing in \eqref{theorem:ulamanddistributivity 2d} that a ${<}(\nu^{<\gamma})^+$-complete Boolean algebra $\B$ is $(\gamma,\nu)$-distributive if and only if \emph{Cut} does not have a winning strategy in the game $\GG_\nu(X,\B,\gamma)$.

\begin{theorem}\label{theorem:ulamanddistributivity}
  Let $\B$ be a poset, $\gamma$ a limit ordinal, $\nu$ a regular cardinal or $\nu=\infty$, and $X\in Q$.
  Then, the following hold.
  \begin{enumerate}
    \item 
    \label{theorem:ulamanddistributivity 1}
    If \emph{Cut} does not have a winning strategy in the game $\GG_\nu(X,\B,\gamma)$, then $\B$ is $(\gamma,\nu)$-distributive with respect to $X$.
    
    \item 
    \label{theorem:ulamanddistributivity 2}
    If $\B$ is $(\gamma,\nu)$-distributive with respect to $X$ and either 
    \begin{enumerate} 
    \item 
    \label{theorem:ulamanddistributivity 2a}
    $\nu=\infty$, 
    \end{enumerate} 
    or $\gamma$ is a cardinal and either 
    
    \begin{enumerate} 
    \setcounter{enumii}{1} 
    \item 
    \label{theorem:ulamanddistributivity 2b}
    $\nu^{<\gamma}=\gamma$, 
    
    \item 
    \label{theorem:ulamanddistributivity 2c}
    $\nu^{<\gamma}=\nu$ and $\B$ is ${<}\gamma$-complete,
    
     \item 
     \label{theorem:ulamanddistributivity 2d}
     $\B$ is ${<}(\nu^{<\gamma})^+$-complete, or 
     
    \item 
    \label{theorem:ulamanddistributivity 2e}
    $\nu^\beta<\nu^\gamma$ for all $\beta<\gamma$ and 
     $\B$ is a ${<}(\nu^{<\gamma})$-complete Boolean algebra, 
 
     \end{enumerate} 
     then \emph{Cut} does not have a winning strategy in the game $\GG_\nu(X,\B,\gamma)$.

    \item \label{theorem:ulamanddistributivity 3}
    Both \eqref{theorem:ulamanddistributivity 1} and \eqref{theorem:ulamanddistributivity 2} hold for $\GG_\nu(X,\B,{<}\gamma)$ and uniform $({<}\gamma,\nu)$-distributivity as well, in the obvious sense.
  \end{enumerate}
\end{theorem}
\begin{proof}
\eqref{theorem:ulamanddistributivity 1}: 
Any sequence of maximal antichains of $\B$ below $X$, each of size at most~$\nu$ (or of arbitrary size in case $\nu=\infty$), witnessing $\B$ to not be $(\gamma,\nu)$-distributive with respect to $X$ can be identified with a strategy for \emph{Cut} in the game $\GG_\nu(X,\B,\gamma)$, and the nonexistence of a positive branch through such a sequence corresponds to the fact that \emph{Choose} cannot win against this strategy, which means that it is in fact a winning strategy, as desired.

\eqref{theorem:ulamanddistributivity 2}: 
Assume that $\B$ is $(\gamma,\nu)$-distributive with respect to $X$. Assume for a contradiction that \emph{Cut} did have a winning strategy $\sigma$ in the game $\GG_\nu(X,\B,\gamma)$.
    
\eqref{theorem:ulamanddistributivity 2a} will be verified within the proof of \eqref{theorem:ulamanddistributivity 2c}. 

    If $\nu^{<\gamma}=\gamma$ in \eqref{theorem:ulamanddistributivity 2b}, then we can construct a $\gamma$-sequence of maximal antichains of $\B$ below $X$, each of size at most $\nu$, consisting of all moves of \emph{Cut} that could possibly come up in any run of $\GG_\nu(X,\B,\gamma)$ in which they follow their winning strategy~$\sigma$, which is a total of $\nu^{<\gamma}=\gamma$-many maximal antichains of $\B$ below $X$, and use $(\gamma,\nu)$-distributivity with respect to $X$ to obtain a positive branch through these, which yields a way for \emph{Choose} to win while \emph{Cut} is following their supposed winning strategy, which is a contradiction.    
    
   If $\nu^{<\gamma}=\nu$ in \eqref{theorem:ulamanddistributivity 2c} (this is also the case if $\nu=\infty$), we inductively construct a $\gamma$-sequence $\langle W_\alpha\mid\alpha<\gamma\rangle$ of maximal antichains of $\B$ below~$X$, each of size at most~$\nu$ (or of arbitrary size in case $\nu=\infty$), as follows: Let~$W_0$ be the first move of \emph{Cut} according to $\sigma$. Given $W_\alpha$, and a possible choice $x$ of \emph{Choose} in their $\alpha^\textrm{th}$ move in a run of the game $\GG_\nu(X,\B,\gamma)$ in which \emph{Cut} plays according to their winning strategy~$\sigma$, let $X_\alpha$ be the set of all such $x$, and let $Y_x$ be the response of $\sigma$ to $x$ being chosen at stage~$\alpha$. Let $W_{\alpha+1}=\{x\,\land\,y\mid x\in X_\alpha\,\land\,y\in Y_x\}$. For limit ordinals~$\alpha$, define~$W_\alpha$ similarly, letting $X_\alpha$ be the set of nonzero greatest lower bounds of the possible first $\alpha$-many choices of \emph{Choose} in runs of the game $\GG_\nu(X,\B,\gamma)$ in which \emph{Cut} plays according to $\sigma$ (if $\nu=\infty$, we let $X_\alpha$ be a maximal antichain of $\B$ of elements below sequences of possible first $\alpha$-many choices of \emph{Choose}, allowing us to drop the completeness assumption on $\B$). Note that by our assumption that $\nu^{<\gamma}=\nu$, these antichains will always have size at most~$\nu$. Use $(\gamma,\nu)$-distributivity with respect to~$X$ to obtain a positive branch through the sequence of $W_\alpha$'s, which yields a way for \emph{Choose} to win while \emph{Cut} is following their supposed winning strategy, which is a contradiction.
   
      Suppose that \emph{Cut} has a winning strategy in the game $\GG_\nu(X,\B,\gamma)$ in either \eqref{theorem:ulamanddistributivity 2d} or  \eqref{theorem:ulamanddistributivity 2e}. 
   Then they also have a winning strategy in $\GG_{\nu^{<\gamma}}(X,\B,\gamma)$. 
   Since $(\nu^{<\gamma})^{<\gamma}=\nu^{<\gamma}$, \eqref{theorem:ulamanddistributivity 2c} shows that $\B$ is not $(\gamma,\nu^{<\gamma})$-distributive. 
   Then $\B$ is not $(\gamma,\nu)$-distributive by Lemma \ref{lemma: eq distributive}  \eqref{lemma: eq distributive 2} or \eqref{lemma: eq distributive 3} for $\alpha=\gamma$. 
    
    \eqref{theorem:ulamanddistributivity 3} follows by exactly the same arguments as \eqref{theorem:ulamanddistributivity 1} and \eqref{theorem:ulamanddistributivity 2} using the instances of Lemma \ref{lemma: eq distributive} about uniform distributivity. 
\end{proof}


\begin{definition}\label{definition:idealdistributivity}
  Let $I$ be an ideal on a regular and uncountable cardinal $\kappa$. Let $\gamma$ be a limit ordinal and let $\nu$ be a regular cardinal, or $\nu=\infty$. 
  $I$ is \emph{$({\le}\gamma,\nu)$-distributive} if whenever $X\in I^+$ and $\langle W_\alpha\mid\alpha<\gamma\rangle$ is a sequence of $I$-partitions of $X$, each of size at most $\nu$, or of arbitrary size in case $\nu=\infty$, then there is a sequence $\langle X_\alpha\mid\alpha<\gamma\rangle$ so that for each $\alpha<\gamma$, $X_\alpha\in W_\alpha$ and for every $\delta<\gamma$, $\bigcap_{\epsilon<\delta}X_\epsilon\in I^+$, and $\bigcap\{X_\alpha\mid\alpha<\gamma\}\ne\emptyset$ in the above. We call such a sequence $\langle X_\alpha\mid\alpha<\gamma\rangle$ a \emph{(nonempty) branch through $\langle W_\alpha\mid\alpha<\gamma\rangle$}.
\end{definition}

The proof of the following theorem essentially proceeds like the proof of Theorem~\ref{theorem:ulamanddistributivity} \eqref{theorem:ulamanddistributivity 1} and \eqref{theorem:ulamanddistributivity 2a}-\eqref{theorem:ulamanddistributivity 2c}, and we will thus omit presenting the argument.

\begin{theorem}\label{theorem:ulamandidealdistributivity}
  Let $I$ be an ideal on a regular and uncountable cardinal $\kappa$, let $\nu$ be a cardinal or $\nu=\infty$, and let $X\in I^+$.
  \begin{enumerate}
    \item 
       \label{theorem:ulamandidealdistributivity 1}
            if $\gamma<\kappa$ is a limit ordinal and \emph{Cut} does not have a winning strategy in the game $\GG_\nu(X,I,{\le}\gamma)$, then $I$ is $({\le}\gamma,\nu)$-distributive with respect to $X$.
        \item 
        \label{theorem:ulamandidealdistributivity 2}
        If $I$ is ${<}\gamma$-complete and $({\le}\gamma,\nu)$-distributive with respect to $X$, and either $\nu=\infty$, $\gamma$ is a cardinal and $\nu^{<\gamma}=\gamma$, or $\nu^{<\gamma}=\nu$, then \emph{Cut} does not have a winning strategy in the game $\GG_\nu(X,I,{\le}\gamma)$.
  \end{enumerate}
\end{theorem}

Recall that a nonprincipal ideal $I$ is \emph{precipitous} if its generic ultrapower is forced to be wellfounded. 
It is a well-known standard result that (in our above terminology) an ideal $I$ is precipitous if and only if it is $({\le}\omega,\infty)$-distributive (see for example \cite[Lemma 22.19]{MR1940513}). It thus follows by the above that precipitousness of an ideal can be described via the non-existence of winning strategies for \emph{Cut} in suitable cut and choose games. We will say more about the relationship between precipitousness and cut and choose games in Section \ref{section:ulamandprecipitous} below.

With respect to Footnote \ref{footnote:WC}, let us also remark that a ${<}\kappa$-complete ideal $I\supseteq\bd_\kappa$ is a \emph{WC ideal} (as defined by Johnson in \cite{MR853844}) if and only if $I$ is uniformly $({<}\kappa,\kappa)$-distributive.

\section{Banach-Mazur games and strategic closure}\label{section:ulamandprecipitous}

%

In this section, we want to show how winning strategies for Banach-Mazur games on partial orders relate to winning strategies for certain cut and choose games.

\begin{definition}[Banach-Mazur games]
Let $\kappa$ be a regular uncountable cardinal, let $I$ be a monotone family on $\kappa$, let $\gamma$ be a limit ordinal, and let $\B$ be a poset with domain $Q$.
  \begin{itemize}
      \item If $\gamma<\kappa$, let $\PP(I,\gamma)$ denote the following game of length $\gamma$. Two players, \emph{Empty} and \emph{Nonempty} take turns to play $I$-positive sets, forming a $\subseteq$-decreasing sequence, with \emph{Empty} starting the game and with \emph{Nonempty} playing first at each limit stage of the game. If at any limit stage $\delta<\gamma$, \emph{Nonempty} cannot make a valid move, then \emph{Empty} wins and the game ends. \emph{Nonempty} wins if the game proceeds for $\gamma$-many rounds and the intersection of the sets that were played is nonempty. Otherwise, \emph{Empty} wins.
    \item $\PP(\B,\gamma)$ denotes the following game of length $\gamma$ on the poset $\B$. Two players, \emph{Empty} and \emph{Nonempty} take turns to play elements of $Q$, forming a $\le$-decreasing sequence, with \emph{Empty} starting the game and with \emph{Nonempty} playing first at each limit stage of the game. If at any limit stage $\delta<\gamma$, \emph{Nonempty} cannot make a valid move, then \emph{Empty} wins and the game ends. \emph{Nonempty} wins if the game proceeds for $\gamma$-many rounds and the collection of the sets that were played has a lower bound in $Q$.
    \item We let $\PP^+(I,\gamma)$ denote the game $\PP(\Pow(\kappa)/I,\gamma)$.
  \end{itemize}
\end{definition}

Clearly, if \emph{Nonempty} has a winning strategy in a game $\PP^+(I,\gamma)$ for some $\gamma<\kappa$ and $I$ contains all singletons, then the same strategy makes them win $\PP(I,\gamma)$. Let us observe that by a similar proof to that of Observation \ref{observation:cutwinsulam}, it easily follows that if $\kappa\le 2^\gamma$, then \emph{Empty} has a winning strategy in the game $\PP^+(I,\gamma)$ for any monotone family $I$ on $\kappa$ that contains all singletons. Note that an ideal $I$ is \emph{precipitous} if and only if \emph{Empty} has no winning strategy in the game $\PP(I,\omega)$, and that a poset $\B$ is ${<}\gamma^+$-strategically closed\footnote{See \cite[Definition 5.15]{MR2768691}. Note that any $(\gamma+1)$-strategically closed poset is already ${<}\gamma^+$-strategically closed by an easy inductive argument.} if and only if \emph{Nonempty} has a winning strategy in the game $\PP(\B,\gamma)$. We recall a classical result, which is verified when $\kappa=\omega_1$ as \cite[Theorem 4]{gjm}, and it is easy to see that the proof of \cite[Theorem~4]{gjm} in fact shows the following, replacing~$\omega_1$ by an arbitrary regular and uncountable cardinal $\kappa$. This parallels Observation \ref{observation:genericallyulam} and the comments preceding it.

\begin{theorem}[Galvin, Jech, Magidor]\label{theorem:gjm}
  Let $\kappa$ be a regular and uncountable cardinal, and let $\gamma<\kappa$ be regular. If we L{\'e}vy collapse a measurable cardinal above $\kappa$ to become $\kappa^+$, then in the generic extension, there is a uniform normal ideal $I$ on $\kappa^+$ such that \emph{Nonempty} has a winning stategy in the game $\PP^+(I,\gamma)$.
\end{theorem}

In the following, we want to compare the above games with the cut and choose games from our earlier sections. When $\gamma=\omega$, \eqref{char precipitous 1}  and \eqref{char precipitous 2} below are essentially due to Jech in \cite{jechabstract,jech}. For larger $\gamma$, 
\eqref{char precipitous 2} below follows from \cite[Theorem on Page~718]{foreman} and \cite[Theorem 1.4]{dobrinen} (we presented the latter in Theorem \ref{theorem:ulamanddistributivity}): That is, in his~\cite{foreman}, 
Matthew Foreman showed that \emph{Empty} not winning $\PP(\B,\gamma)$ is equivalent to the $(\gamma,\infty)$-distributivity of $\B$. We will provide an argument that directly connects these types of games.

\begin{theorem}
\label{char precipitous} 
Let $\kappa$ be a regular uncountable cardinal, let $I$ be an ideal on $\kappa$, let $\gamma<\kappa$ be a limit ordinal, and let $\B$ be a poset with domain $Q$. Then, the following hold:
  \begin{enumerate}
    \item \label{char precipitous 1} 
    \emph{Empty} wins $\PP(I,\gamma)$ if and only if \emph{Cut} wins $\GG_\infty(X,I,{\le}\gamma)$ for some $X\in I^+$.
    \item \label{char precipitous 2} 
    \emph{Empty} wins $\PP(\B,\gamma)$ if and only if \emph{Cut} wins $\GG_\infty(X,\B,\gamma)$ for some $X\in Q$.
  \end{enumerate}
\end{theorem}
\begin{proof}
  Let us provide a proof of \eqref{char precipitous 1}, and remark that \eqref{char precipitous 2} is verified in complete analogy.
  By Theorem \ref{theorem:ulamandidealdistributivity}, \emph{Cut} having a winning strategy in the game $\GG_\infty(X,I,{\le}\gamma)$ is equivalent to $I$ not being $({\le}\gamma,\infty)$-distributive with respect to $X$.
  
  Assuming that $I$ is not $({\le}\gamma,\infty)$-distributive with respect to $X$, we pick a sequence $\langle W_i\mid i<\gamma\rangle$ of $I$-partitions of $X$ witnessing this. Let us describe a winning strategy for \emph{Empty} in the game $\PP(I,\gamma)$. In their first move, let \emph{Empty} play the set $x_0=X$. At any stage $i<\gamma$, given the last move $y\in I^+$ of \emph{Nonempty}, pick $x_i\in W_i$ such that $y\cap x_i\in I^+$, which exists by the maximality of $W_i$. Let \emph{Empty} play $x_i$. It follows that $\bigcap_{i<\gamma}x_i=\emptyset$, as desired.
  
  On the other hand, assume that \emph{Empty} has a winning strategy $\sigma$ in the game $\PP(I,\gamma)$. Let $x_0$ be the first move of \emph{Empty} according to $\sigma$. We will describe a winning strategy for \emph{Cut} in the game $\GG_\infty(x_0,I,{\le}\gamma)$, making use of an auxiliary run of $\PP(I,\gamma)$ according to $\sigma$. 
  Given a play of $\PP(I,\gamma)$ in which the moves of \emph{Empty} are $\langle x_i\mid i<j\rangle$ for some $j<\gamma$, and \emph{Nonempty} is to move next, for every possible next move $q\in I^+$ of \emph{Nonempty}, $\sigma$ has a response $r\subseteq q$ in~$I^+$, which provides us with a dense set $D_j$ of such responses $r$ below $\bigcap_{i<j}x_i$ in $\Pow(\kappa)/I$. Noting that maximal antichains in $\Pow(\kappa)/I$ are exactly $I$-partitions, let $\bar W_j\subseteq D_j$ be an $I$-partition of $\bigcap_{i<j}x_i$. Let \emph{Cut} play an $I$-partition $W_j$ of $x_0$ extending $\bar W_j$ in their $j^{\textrm{th}}$ move. \emph{Choose} will pick some element $w_j\in\bar W_j$, and we let \emph{Nonempty} play some $I$-positive $q_j$ in their next move, such that \emph{Empty} answers this by playing $x_j=w_j$ in their next move, according to $\sigma$.  In this way, all choices of \emph{Choose} are also moves of \emph{Empty}, hence $\bigcap_{i<\gamma}w_i=\emptyset$, since \emph{Empty} is following their winning strategy $\sigma$.
  \end{proof}

The next theorem will show that we in fact obtain instances of equivalent Banach-Mazur games and cut and choose games.
The forward direction when $\gamma=\omega$ in Item~\eqref{theorem:precipitousyieldsulam 2} below is due to Jech in \cite{jech}, and the reverse direction of \eqref{theorem:precipitousyieldsulam 2} for $\gamma=\omega$ is due to Veličković in \cite{boban}. The full proof of \eqref{theorem:precipitousyieldsulam 2} below is due to Dobrinen \cite[Theorem~29]{dobrinen2}.

\begin{theorem}\label{theorem:precipitousyieldsulam}
  Let $\kappa$ be a regular uncountable cardinal, let $I$ be an ideal on $\kappa$, let $\gamma<\kappa$ be a limit ordinal, and let $\B$ be a poset with domain $Q$. Then, the following hold:
  \begin{enumerate}
    \item 
    \label{theorem:precipitousyieldsulam 1}
    \emph{Nonempty} wins $\PP(I,\gamma)$ if and only if \emph{Choose} wins $\GG_\infty(X,I,{\le}\gamma)$ for all $X\in I^+$.
    \item 
    \label{theorem:precipitousyieldsulam 2}
    \emph{Nonempty} wins $\PP(\B,\gamma)$ if and only if \emph{Choose} wins $\GG_\infty(X,\B,\gamma)$ for all $X\in Q$.
  \end{enumerate}
\end{theorem}
\begin{proof}
  We provide a proof of \eqref{theorem:precipitousyieldsulam 1}, and remark that \eqref{theorem:precipitousyieldsulam 2} is verified in complete analogy.
  For the forward direction, let $\sigma$ be a winning strategy for \emph{Nonempty} in $\PP(I,\gamma)$, and let $X\in I^+$. We describe a winning strategy for \emph{Choose} in $\GG_\infty(X,I,{\le}\gamma)$, making use of an auxiliary run of $\PP(I,\gamma)$ according to $\sigma$.  
  Suppose that \emph{Cut} starts the game by playing an $I$-partition $W_0$ of $X$. Let \emph{Empty} play $x_0=X$, and let $y_0$ be the response of $\sigma$. Using the maximality of $W_0$, let \emph{Choose} pick $w_0\in W_0$ such that $w_0\cap y_0\in I^+$ as their next move. 
  At any stage $0<i<\gamma$, assume \emph{Cut} plays an $I$-partition $W_i$ of $X$, and let $y_i$ be the last move of \emph{Nonempty} according to $\sigma$. Let \emph{Choose} pick $w_i\in W_i$ such that $w_i\cap y_i\in I^+$ as their next move. Let \emph{Empty} play $w_i\cap y_i$, and let \emph{Nonempty}  respond with $y_{i+1}$ using $\sigma$. At limit stages, let \emph{Nonempty} make a move according to $\sigma$.
Since $y_{i+1}\subseteq w_i$, we have $\bigcap_{i<\gamma}w_i \supseteq \bigcap_{i<\gamma}y_i \neq\emptyset$, showing that we have indeed described a winning strategy for \emph{Choose}, as desired.

 For the reverse direction, suppose that \emph{Empty} starts a run of the game $\PP(I,\gamma)$ by playing some $x_0\in I^+$. Let $\sigma$ be a winning strategy for \emph{Choose} in the game $\GG_\infty(x_0,I,{\le}\gamma)$. We can identify $\sigma$ with a function $F$ which on input $\langle W_i\mid i\le\delta\rangle$ for some $\delta<\gamma$ considers the partial run in which the moves of \emph{Cut} are given by the $W_i$, the moves of \emph{Choose} at stages below $\delta$ are given by the strategy $\sigma$, and $F(\langle W_i\mid i\le\delta\rangle)$ produces a response $w_\delta\in W_\delta$ for \emph{Choose} to this partial run. We describe a winning strategy for \emph{Nonempty} in the game $\PP(I,\gamma)$, making use of an auxiliary run of $\GG_\infty(x_0,I,{\le}\gamma)$ according to $\sigma$.
  
  For the first move, consider the set \[\Sigma_\emptyset=\{F(\langle W\rangle)\mid W\textrm{ is an }I\textrm{-partition of }x_0\},\]
  and note that there is an $I$-positive $y_0\subseteq x_0$ such that $\Pow(y_0)\cap I^+\subseteq\Sigma_\emptyset$, for otherwise the complement of $\Sigma_\emptyset$ is dense in $I^+$ below $x_0$, and hence there is an $I$-partition $W$ of $x_0$ that is disjoint from $\Sigma_\emptyset$, however $F(\langle W\rangle)\in W\cap\Sigma_\emptyset$, which is a contradiction. Let \emph{Nonempty} pick such a $y_0$ as their response to \emph{Empty}'s first move $x_0$.
  
  In the next round, suppose that \emph{Empty} plays $x_1\subseteq y_0$. Let \emph{Cut} play an $I$-partition $W_0$ of $x_0$ such that $F(\langle W_0\rangle)=x_1$ as their first move in the game $\GG_\infty(x_0,I,{\le}\gamma)$. Consider the set \[\Sigma_{\langle W_0\rangle}=\{F(\langle W_0,W \rangle)\mid W\textrm{ is an }I\textrm{-partition of }x_0\}.\]
  As before, there is $y_1\subseteq x_1$ in $I^+$ such that $\Pow(y_1)\cap I^+\subseteq\Sigma_{\langle W_0\rangle}$, and we let \emph{Nonempty} respond with such $y_1$.
  We proceed in the same way at arbitrary successor stages.
  At any limit stage $0<i<\gamma$, let $\vec W=\langle W_j\mid j<i\rangle$, and let \emph{Nonempty} pick $y_i$ such that $\Pow(y_i)\cap I^+\subseteq\Sigma_{\vec W}=\{F(\vec W^\smallfrown \langle W\rangle)\mid W\textrm{ is an }I\textrm{-partition of }x_0\}$ by an argument as above.
  
  In this way, the choices of \emph{Choose} are exactly the choices of \emph{Empty} in the above, and hence their intersection is nonempty, for \emph{Choose} was following their winning strategy $\sigma$. This shows that we have just described a winning strategy for \emph{Nonempty} in the game $\PP(I,\gamma)$, as desired.
\end{proof}

We showed in Theorem \ref{theorem:ulamanddistributivity} that a poset $\B$ is $(\gamma,\infty)$-distributive if and only if for all $X\in Q$, \emph{Cut} does not win the game $\GG_\infty(X,\B,\gamma)$. 
The next characterization follows from Theorem \ref{theorem:precipitousyieldsulam}, since a poset $\B$ is ${<}\gamma^+$-strategically closed (by the very definition of this property) if and only if \emph{Choose} has a winning strategy in $\PP(\B,\gamma)$. 

\begin{corollary} 
\label{strategic closure} 
  Let $\kappa$ be a regular uncountable cardinal, let $I$ be an ideal on $\kappa$, let $\gamma<\kappa$ be a limit ordinal, and let $\B$ be a poset with domain $Q$. Then, $\B$ is ${<}\gamma^+$-strategically closed if and only if \emph{Choose} wins $\GG_\infty(X,\B,\gamma)$ for all $X\in Q$. 
\end{corollary} 

Let us close with some complementary remarks on the games studied in this section. 
We first argue that allowing for arbitrary large partitions is important in the above characterisations of precipitous ideals. For instance, a restriction to partitions of size ${<}\kappa$ does not lead to equivalent games. 
To see this, note that assuming the consistency of a measurable cardinal and picking some regular cardinal $\gamma$ below, it is consistent to have an ideal $I$ on a cardinal $\kappa$ such that \emph{Choose} has a winning strategy in the game $\GG_{{<}\kappa}(\kappa,I,{\le}\gamma)$ for any $\gamma<\kappa$, but $I$ is not precipitous. Simply take $I=\bd_\kappa$ when $\kappa$ is either measurable, or in a model obtained from Theorem \ref{theorem:gjm}. It is well-known that the bounded ideal is never precipitous (see \cite[Page 1]{gjm}).

\medskip

This can also hold for normal ideals. 
For example, work in a model of the form $L[U]$ with a measurable cardinal $\kappa$ with a solitary normal ultrafilter $U$ on $\kappa$. 
Let $J$ be the ideal on $\kappa$ dual to $U$. 
The cardinal $\kappa$ is also completely ineffable, and we let $I$ be the completely ineffable ideal on $\kappa$, as introduced by Johnson in \cite{MR918427} (see also \cite{pp}). 
Since $J$ equals the measurable ideal on $\kappa$, that is defined as the intersection of the complements of all normal ultrafilters on $\kappa$ in \cite{pp}, we have $I\subseteq J$ by \cite[Theorem 1.4 (5) and Theorem 1.5 (11)]{pp}. 
As in Observation \ref{observation:measurable}, \emph{Choose} has a winning strategy in the game $\GG_{{<}\kappa}(\kappa,I,\gamma)$ whenever $\gamma<\kappa$. But by a result of Johnson \cite[Theorem 1.6]{MR918427}, the completely ineffable ideal is never precipitous, that is \emph{Empty} has a winning strategy in $\PP(I,\omega)$.

\medskip

In general, it is harder for \emph{Cut} to win $\GG_\infty(X,I,{\le}\omega)$ than to win $\GG_\infty(X,I,\omega)$. 
To see this, note that for any precipitous ideal $I$ on $\omega_1$, 
\emph{Cut} does not win $\GG_\infty(X,I,{\le}\omega)$ by Theorem \ref{char precipitous} (1). 
However, \emph{Cut} wins $\GG_\omega(X,I,\omega)$ by Observation \ref{observation:generalizedtrivialforcut}. 

\medskip

Building on results of Galvin, Jech and Magidor~\cite{gjm}, Johnson \cite[Theorem 4]{MR853844} shows that for $\kappa<\aleph_\omega$,\footnote{The proof would also work assuming $\lambda^\omega<\kappa$ whenever $\lambda<\kappa$.} if $J\supseteq \bd_\kappa$ is a ${<}\kappa$-complete ideal  on $\kappa$ such that \emph{Nonempty} wins the game $\PP(J,\omega)$, then $J$ is $(\omega,\kappa)$-distributive. 
Let $I$ be a normal precipitous ideal on $\omega_1$.\footnote{Note that the usual construction of a precipitous ideal on $\omega_1$ yields a normal precipitous ideal. See for example \cite[Theorem 22.33]{MR1940513}.}  
Since $I$ is not $(\omega,\omega_1)$-distributive, there exists $X\in I^+$ such that \emph{Choose} does not win $\GG_\infty(X,I,{\le}\omega)$ by the above and by Theorem \ref{theorem:precipitousyieldsulam} \eqref{theorem:precipitousyieldsulam 1}. 
By the combination of these two observations, it is consistent that there exists a normal ideal $K$ on $\omega_1$ such that $\GG_\infty(\omega_1,K,{\le}\omega)$ is undetermined. 

\medskip

This is still possible for $\aleph_2$. 
To see this, note that Shelah has shown that precipitousness of $I$ does not imply $(\omega,\kappa)$-distributivity of $I$ even if $\kappa=\omega_2$, $\CH$ holds and $I$ is normal (see \cite[Comments before Theorem~4]{MR853844}, \cite[Theorem 2]{MR853844} and \cite[Theorem 6.4 (2)]{shelah1981iterated}).
In this situation, there exists some $X\in I^+$ such that \emph{Cut} wins $\GG_{\omega_2}(X,I,\omega)$ by 
Theorem \ref{theorem:ulamanddistributivity}, but for all $Y\in I^+$, \emph{Cut} does not win $\GG_\infty(Y,I,{\le}\omega)$ by Theorem \ref{theorem:precipitousyieldsulam}. 
Using \cite[Theorem 4]{MR853844} and Theorem \ref{theorem:precipitousyieldsulam} \eqref{theorem:precipitousyieldsulam 1} as above, there exists $X\in I^+$ such that \emph{Choose} does not win $\GG_\infty(X,I,{\le}\omega)$. 
Hence it is consistent that there exists a normal ideal $K$ on $\omega_2$ such that $\GG_\infty(\omega_2,K,{\le}\omega)$ is undetermined.

\section{Final remarks and open questions}\label{section:questions}

We have seen that for most cut and choose games, the existence of a winning strategy for \emph{Cut} has a precise characterization as in Figure \ref{figure characterisations} below (where (*) indicates that the stated equivalence is only known to hold under certain additional completeness and cardinal arithmetic assumptions). 
Let $\kappa$ be an uncountable regular cardinal, $\gamma<\kappa$ a limit ordinal, $\nu$ a regular cardinal and $I\supseteq \bd_\kappa$ a ${<}\kappa$-complete ideal on $\kappa$. 
For the proofs of the first three rows of Figure \ref{figure characterisations}, see Observations \ref{observation:cutwinsulam}, \ref{cut does not win the weak Ulam game for large kappa} and \ref{observation:cutwinsidealulam}, and for the remaining ones 
Observation \ref{observation:generalizedtrivialforcut} \eqref{observation:generalizedtrivialforcut 2} and \eqref{observation:generalizedtrivialforcut 1}, Theorems  \ref{theorem:ulamanddistributivity}, \ref{theorem:ulamandidealdistributivity} and  \ref{char precipitous} \eqref{char precipitous 1}. 
Recall the equivalence between $\UU$-games (where \emph{Cut} presents disjoint partitions of size $\nu<\kappa$) and $\GG$-games proved in Observation \ref{observation:stargeneralizes}. 

\begin{figure}[H] 
\begin{tabular}{ l l } 
\emph{Cut} wins & Characterization \\ 
\hline 
$\UU(\kappa,{\le}\gamma)$, $\UU(\kappa,I,{\le}\gamma)$ & $\kappa\leq 2^{<\gamma}$ \\ 
$\UU(\kappa,\gamma)$, $\UU(\kappa,I,\gamma)$ & $\kappa\leq 2^{|\gamma|}$ \\ 
$\UU(\kappa,\bd_\kappa,{<}\kappa)$ & $\kappa$ is not weakly compact \\ 
$\UU_\nu(\kappa,I,{\le}\gamma)$ & $\kappa\le\nu^{<\gamma}$ \\ 
$\UU_\nu(\kappa,I,\gamma)$ & $\kappa\le\nu^{|\gamma|}$ \\ 
$\forall X\,\GG_\nu(X,I,{<}\gamma)$ & $I$ is not uniformly $({<}\gamma,\nu)$-distributive (*) \\ 
$\forall X\,\GG_\nu(X,I,{\le}\gamma)$ & $I$ is not $({\le}\gamma,\nu)$-distributive (*)  \\ 
$\forall X\,\GG_\nu(X,I,\gamma)$ & $I$ is not $(\gamma,\nu)$-distributive (*)\\ 
$\forall X\,\GG_\infty(X,I,{\le}\gamma)$ & $I$ is not precipitous\\ 
\hline 
\end{tabular} 
\medskip 
\caption{Characterizations of the existence of winning strategies for \emph{Cut} in various cut and choose games. } 
\label{figure characterisations} 
\end{figure} 

Regarding winning strategies for \emph{Choose}, we have seen in Corollary \ref{strategic closure} that \emph{Choose} wins $\GG_\infty(X,\B,\gamma)$ for all $X\in Q$ if and only if  $\B$ is ${<}\gamma^+$-strategically closed. 
However, it is often much harder to characterize the existence of winning strategies for \emph{Choose}. 
For instance, it is open in many cases whether the existence of winning strategies for \emph{Choose} in different cut and choose games can be separated. 


\begin{question}
\label{Question - separate games for choose} 
  Is it possible to separate the existence of winning strategies for \emph{Choose} in  $\UU(\kappa,{\le}\gamma)$ and $\GG_\infty(\kappa,\bd_\kappa,\gamma)$ for some limit ordinal $\gamma<\kappa$? 
\end{question}

Besides these two extreme cases, the previous question is also open for games in between the above ones, such as $\UU(\kappa,\gamma)$, or the variant of $\UU(\kappa,\gamma)$ with a different (finite or infinite) number of elements required in the final intersection. 
An obvious candidate for a model to answer Question \ref{Question - separate games for choose} positively
would be the model obtained in the proof of Theorem~\ref{theorem:winulamatsmallinaccessible}.

\begin{question}\label{question:intermediate}
  In the model obtained in the proof of Theorem \ref{theorem:winulamatsmallinaccessible}, where $\kappa$ is inaccessible but not weakly compact, and in which \emph{Choose} has a winning strategy in the game $\UU(X,I,{\le}\gamma)$ whenever $\gamma<\kappa$ and $X\in I^+$, with $I$ being a hopeless ideal as obtained from the generic measurability of $\kappa$ in that model, does \emph{Nonempty} have a winning strategy in the precipitous game $\PP(I,\gamma)$ (and hence in the game $\GG_\infty(X,I,{\le}\gamma)$ for every $X\in I^+$ by Theorem \ref{theorem:precipitousyieldsulam})?
\end{question}

A related natural question regarding the ideal games studied in Section \ref{section:ideal} is whether the existence of a winning strategy for \emph{Choose} can ever depend on the choice of ideal. We formulate our question for ideal games of the form $\UU(\kappa,I,\gamma)$ in the below, however analogous questions could clearly be asked regarding all sorts of variants of these games that we study in our paper.

\begin{question}\label{question:differentideals}
Is it consistent that there exist ${<}\kappa$-closed ideals $I,J \supseteq \bd_\kappa$ on $\kappa$ so that \emph{Choose} has a winning strategy for $\UU(\kappa,I,\gamma)$, but not for $\UU(\kappa,J,\gamma)$? \end{question}

Note that by Observation \ref{observation:stronglycompact}, this cannot happen if $\kappa$ is $2^\kappa$-strongly compact.
This question is equivalent to the question whether the existence of winning strategies for \emph{Choose} in any of the ideal cut and choose games introduced in this paper can depend on the choice of starting set $X\in I^+$, when $I\supseteq\bd_\kappa$ is a ${<}\kappa$-complete ideal on $\kappa$. 
If Question \ref{question:differentideals} had a positive answer, then we could easily form an ideal $K$ that is generated by isomorphic copies of such ideals $I$ and $J$ on two disjoint subsets of $\kappa$ such that the existence of a winning strategy of \emph{Choose} in the game $\UU(X,K,\gamma)$ depends on the starting set $X$. 
In the other direction, if the existence of winning strategies for \emph{Choose} in some ideal cut and choose game related to $I$ can depend on the choice of starting set $X\in I^+$, then we can consider the games with respect to the ideals obtained by restricting $I$ to those starting sets, exactly one of which will be won by \emph{Choose}.

\medskip


In Section \ref{section:smallinaccessibles}, we have seen that it is consistent that \emph{Choose} wins $\UU(\kappa,\bd_\kappa,\gamma)$ for a small inaccessible cardinal $\kappa$. 
The cardinal studied there is not weakly compact, however Mahlo. 
We do not know if the latter is necessary. 

\begin{question} 
Is it consistent that \emph{Choose} wins $\UU(\kappa,{\le}\omega)$, where $\kappa$ is the least inaccessible cardinal? 
\end{question} 

Theorem \ref{theorem:ulamstrength} and Observation \ref{observation:genericallyulam} show that if $\gamma<\kappa$ is regular and \emph{Choose} wins $\UU(\kappa,{\le}\gamma)$, then $\kappa$ is generically measurable as witnessed by some ${<}\gamma$-closed notion of forcing, which in turn implies that \emph{Choose} wins $\UU(\kappa,\bd_\kappa,{<}\gamma)$. We ask if either of these implications can be reversed:

\begin{question}
  Let $\gamma<\kappa$ be regular.
  \begin{enumerate}
    \item If $\kappa$ is generically measurable as witnessed by some ${<}\gamma$-closed notion of forcing, does it follow that \emph{Choose} wins $\UU(\kappa,{\le}\gamma)$?
    \item If \emph{Choose} wins $\UU(\kappa,\bd_\kappa,{<}\gamma)$, does it follow that $\kappa$ is generically measurable as witnessed by some ${<}\gamma$-closed notion of forcing?
  \end{enumerate}
  This latter question makes sense also if $\gamma=\kappa$:
\begin{enumerate}
\setcounter{enumi}{2}
    \item If \emph{Choose} wins $\UU(\kappa,\bd_\kappa,{<}\kappa)$, does it follow that $\kappa$ is generically measurable as witnessed by some ${<}\kappa$-closed notion of forcing?
  \end{enumerate}
\end{question}

Regarding the characterisations of distributivity in Section \ref{section:distributivity}, we have partially answered a question of Dobrinen \cite[paragraph after Theorem 1.4]{dobrinen}, however the following is still open to some extent.

\begin{question}
  Which degree of completeness of a poset $\B$ is necessary in order to show that the $(\gamma,\nu)$-distributivity of $\B$ implies that \emph{Cut} does not have a winning strategy in the game $\GG_\nu(X,\B,\gamma)$?
\end{question} 
  
Note that ${<}(\nu^{<\gamma})^+$-completeness suffices by Theorem \ref{theorem:ulamanddistributivity}. 
In conjunction with Lemma \ref{lemma: eq distributive}, this shows that the existence of a winning strategy for \emph{Cut} in the games $\GG_\nu(X,\B,\gamma)$ and $\GG_{\nu^{<\gamma}}(X,\B,\gamma)$ 
are equivalent. 
In most cases, an analogous result holds with respect to the existence of winning strategies for \emph{Choose} by Theorem \ref{theorem:increasedsplitting}, however we do not know whether this is the case when $\gamma$ is a limit cardinal such that $\nu^\delta<\nu^{<\gamma}$ whenever $\delta<\gamma$. 
In this direction, Theorem \ref{theorem:increasedsplitting}~\eqref{theorem:increasedsplitting 2} only shows that $\GG_\nu(X,\B,\gamma)$ and $\GG_{<(\nu^{<\gamma})}(X,\B,\gamma)$ are equivalent. 

\begin{question} 
Is the existence of a winning strategy for \emph{Choose} in the games $\GG_\nu(X,\B,\gamma)$ and $\GG_{\nu^{<\gamma}}(X,\B,\gamma)$ equivalent, assuming that $\B$ is ${<}(\nu^{<\gamma})^+$-complete? 
\end{question}

The following questions concern the relationship between cut and choose games and Banach-Mazur games.
The first one is connected with the separation of the existence of winning strategies for \emph{Choose} for different games. 
For instance, is it consistent that \emph{Choose} wins $\UU(\omega_2,{\le\omega})$, but \emph{Cut} wins $\GG_\infty(\omega_2,I,{\le}\omega)$ for arbitrary ${<}\kappa$-complete ideals $I\supseteq\bd_\kappa$? 
This is equivalent to the following question. 

\begin{question}\label{question:noprecipitous}
  Is it consistent that \emph{Choose} wins the game $\UU(\omega_2,{\le}\omega)$, but there are no precipitous ideals on $\omega_2$? 
\end{question}

Similarly, one can ask about situations in which certain cut and choose games are undetermined. 
Note that for all cut and choose games studied in this paper, the existence of a winning strategy for \emph{Choose} implies that there is an inner model with a measurable cardinal. 
Therefore, the statements listed in Figure \ref{figure characterisations} show that all games of length $\omega$ where \emph{Cut} plays partitions of size $\nu<\kappa$ can be simultaneously undetermined, for instance when $\kappa=\omega_2$ and $\CH$ holds. 
Concerning larger partitions, we have seen at the end of Section \ref{section:ulamandprecipitous} that it is consistent  that each of $\GG_\infty(\omega_1,I,{\le}\omega)$ and $\GG_\infty(\omega_2,I,{\le}\omega)$ is undetermined, and in fact, the former holds for any normal precipitous ideal $I$ on $\omega_1$. This leads to the following question, which is left open by the discussion at the end of Section \ref{section:ulamandprecipitous}.

\begin{question}
  Can the game $\GG_\infty(\kappa,I,\omega)$ be undetermined for some ${<}\kappa$-complete ideal $I\supseteq\bd_\kappa$ on $\kappa$?
\end{question}

This leaves open whether the games $\UU(\omega_2,{\le}\omega)$ and $\GG_\infty(\omega_2,I,{\le}\omega)$ can be simultaneously undetermined for some ${<}\kappa$-complete ideal $I\supseteq\bd_\kappa$. This would follow from a positive answer to the next question, which is closely connected to Question~\ref{question:noprecipitous}.

\begin{question}
  Is it consistent that there is a precipitous ideal on $\omega_2$, however \emph{Choose} does not win $\UU(\omega_2,\omega)$?
\end{question}

We would finally like to mention a few more natural question about Banach-Mazur games. 
We defined $\PP(I,\gamma)$ for $\gamma>\omega$ so that at limit stages, \emph{Nonempty} goes first. Can we separate the existence of winning strategies for either player between this game and its variant where we let \emph{Empty} go first at limit stages?
Moreover, the game $\PP^+(I,\gamma)$ depends only on the isomorphism type of $P(\kappa)/I$. 
Is this also the case for $\PP(I,\gamma)$? 
For instance, if $I$ and $J$ are ideals on $\kappa$ with $P(\kappa)/I \cong P(\kappa)/J$, does \emph{Empty} win $\PP(I,\gamma)$ if and only if \emph{Empty} wins $\PP(J,\gamma)$?

\bibliographystyle{plain}
\bibliography{references}

\begin{thebibliography}{10}

\bibitem{MR2768691}
James Cummings.
\newblock Iterated forcing and elementary embeddings.
\newblock In {\em Handbook of set theory. {V}ols. 1, 2, 3}, pages 775--883.
  Springer, Dordrecht, 2010.

\bibitem{dobrinen}
Natasha Dobrinen.
\newblock Games and general distributive laws in {B}oolean algebras.
\newblock {\em Proc. Am. Math. Soc.}, 131(1):309--318, 2003.

\bibitem{dobrinen2}
Nathasha Dobrinen.
\newblock More ubiquitous undetermined games and other results on uncountable
  length games in {B}oolean algebras.
\newblock {\em Note di Matematica}, 27(1):65--83, 2007.

\bibitem{donder1989some}
Hans-Dieter Donder and Jean-Pierre Levinski.
\newblock Some principles related to {C}hang's conjecture.
\newblock {\em Ann. Pure Appl. Logic}, 45(1):39--101, 1989.

\bibitem{foreman}
Matthew Foreman.
\newblock Games played on {B}oolean algebras.
\newblock {\em J. Symbolic Logic}, 48(3):714--723, 1983.

\bibitem{foreman2020games}
Matthew Foreman, Menachem Magidor, and Martin Zeman.
\newblock Games with filters.
\newblock {\em arXiv preprint arXiv:2009.04074}, 2020.

\bibitem{gjm}
Fred Galvin, Thomas Jech, and Menachem Magidor.
\newblock An ideal game.
\newblock {\em J. Symbolic Logic}, 43(2):284--292, 1978.

\bibitem{MR2830435}
Victoria Gitman and Philip~D. Welch.
\newblock Ramsey-like cardinals {II}.
\newblock {\em J. Symbolic Logic}, 76(2):541--560, 2011.

\bibitem{pp}
Peter Holy and Philipp L\"{u}cke.
\newblock Small models, large cardinals, and induced ideals.
\newblock {\em Ann. Pure Appl. Logic}, 172(2), 2021.

\bibitem{MR3800756}
Peter Holy and Philipp Schlicht.
\newblock A hierarchy of {R}amsey-like cardinals.
\newblock {\em Fund. Math.}, 242(1):49--74, 2018.

\bibitem{jechabstract}
Thomas Jech.
\newblock An infinite game played with stationary sets.
\newblock {\em Not. Am. Math. Soc.}, 23(2):A286--A286, 1976.

\bibitem{jech}
Thomas Jech.
\newblock More game-theoretic properties of {B}oolean algebras.
\newblock {\em Ann. Pure Appl. Logic}, 26:11--29, 1984.

\bibitem{MR1940513}
Thomas Jech.
\newblock {\em Set theory}.
\newblock Springer Monographs in Mathematics. Springer-Verlag, Berlin, 2003.
\newblock The third millennium edition, revised and expanded.

\bibitem{MR853844}
C.~A. Johnson.
\newblock Distributive ideals and partition relations.
\newblock {\em J. Symbolic Logic}, 51(3):617--625, 1986.

\bibitem{MR918427}
C.~A. Johnson.
\newblock More on distributive ideals.
\newblock {\em Fund. Math.}, 128(2):113--130, 1987.

\bibitem{MR1994835}
Akihiro Kanamori.
\newblock {\em The higher infinite}.
\newblock Springer Monographs in Mathematics. Springer-Verlag, Berlin, second
  edition, 2003.

\bibitem{evolution}
Akihiro Kanamori and Menachem Magidor.
\newblock The evolution of large cardinal axioms in set theory.
\newblock In Gert~H. M{\"u}ller and Dana~S. Scott, editors, {\em Higher Set
  Theory}, pages 99--275, Berlin, Heidelberg, 1978. Springer, Berlin
  Heidelberg.

\bibitem{kunen}
Kenneth Kunen.
\newblock Saturated ideals.
\newblock {\em J. Symbolic Logic}, 43(1):65--–76, 1978.

\bibitem{nielsen-welch}
Dan~Saattrup Nielsen and Philip~D. Welch.
\newblock Games and {R}amsey-like cardinals.
\newblock {\em J. Symbolic Logic}, 84(1):408–437, 2019.

\bibitem{scheepers}
Marion Scheepers.
\newblock Games that involve set theory or topology.
\newblock {\em Vietnam J. Math.}, 23(2):169--220, 1995.

\bibitem{SchindlerNotiz}
Ralf Schindler.
\newblock Yet another characterization of remarkable cardinals.
\newblock Unpublished note, available at
  \url{https://ivv5hpp.uni-muenster.de/u/rds/alternative_characterization.pdf},
  2017.

\bibitem{MR2817562}
Ian Sharpe and Philip~D. Welch.
\newblock Greatly {E}rd\"os cardinals with some generalizations to the {C}hang
  and {R}amsey properties.
\newblock {\em Ann. Pure Appl. Logic}, 162(11):863--902, 2011.

\bibitem{shelah1981iterated}
Saharon Shelah.
\newblock Iterated forcing and changing cofinalities.
\newblock {\em Israel J. Math.}, 40(1):1--32, 1981.

\bibitem{Sobot}
Boris {\v{S}}obot.
\newblock Games on {B}oolean algebras.
\newblock Ph.D. thesis, University of Novi Sad, 2009.

\bibitem{boban}
Boban Veličković.
\newblock Playful {B}oolean algebras.
\newblock {\em Trans. Am. Math. Soc.}, 296(2):727--740, 1986.

\bibitem{zeman}
Martin Zeman.
\newblock {\em Inner Models and Large Cardinals}.
\newblock De Gruyter, 2011.

\end{thebibliography}

\end{document}